\documentclass[11pt]{amsart}
\usepackage{amscd, amssymb, mathrsfs, mathabx, 
dsfont, xcolor, %mathtools, 
amsmath, tikz-cd, soul}
\usepackage{kotex}
\usepackage[backref=page]{hyperref}   
\usepackage{collectbox}



\input{xy}
\xyoption{all}

\newtheorem{Thm}{Theorem}[section]

\newtheorem{Prop}[Thm]{Proposition}
\newtheorem{Cor}[Thm]{Corollary}
\newtheorem{lemma}[Thm]{Lemma}
\theoremstyle{definition}
\newtheorem{Def}[Thm]{Definition}
\newtheorem{Def/Thm}[Thm]{Definition/Theorem}

\theoremstyle{remark}
\newtheorem{Rmk}[Thm]{Remark}
\newtheorem{example}{Example}

\numberwithin{equation}{section}
\newcommand{\ti }{\times}
\newcommand{\ot }{\otimes}

\newcommand{\Hom }{{\mathrm{Hom}}}
\newcommand{\tr }{{\mathrm{tr}}}

\newcommand{\Spec}{{\mathrm{Spec}}}

\newcommand{\cA}{{\mathcal{A}}}
\newcommand{\cO}{{\mathcal{O}}}
\newcommand{\cM}{{\mathcal{M}}}

\newcommand{\cE}{{\mathcal{E}}}
\newcommand{\cF}{{\mathcal{F}}}
\newcommand{\cH}{{\mathcal{H}}}
\newcommand{\cP}{{\mathcal{P}}}

\newcommand{\cC}{{\mathcal{C}}}

\newcommand{\HH}{{\mathbb H}}

\newcommand{\GG }{{\mathbb G}}

\newcommand{\CC }{{\mathbb C}}
\newcommand{\ZZ }{{\mathbb Z}}
\newcommand{\RR }{{\mathbb R}}

\newcommand{\kb }{{\beta}}
\newcommand{\ka }{{\alpha}}

\newcommand{\ch}{\mathrm{ch}}
\newcommand{\uC}{\underline{C}}

\newcommand{\sA}{\mathscr{A}}

\newcommand{\Ch}{\mathrm{Ch}}
\newcommand{\cB}{\mathcal{B}}

\newcommand{\End}{\mathrm{End}}

\newcommand{\Perf}{\mathrm{Perf}}
\newcommand{\Mod}{\mathrm{Mod}}
\newcommand{\Hoch}{C}
\newcommand{\oHoch}{\overline{C}}
\newcommand{\uHoch}{\underline{C}}
\newcommand{\MC}{\mathrm{MC}}

\newcommand{\Com}{\mathrm{Com}}

\newcommand{\vC}{\text{\rm \v{C}}}
\newcommand{\Om}{\Omega ^{\bullet}}

\newcommand{\Perfd}{D_{\mathrm{Dol}}}

\newcommand{\PreCom}{\Com_{cdg}}
\newcommand{\fU}{\mathfrak{U}}

\newcommand{\MFdg}{D_{dg}}
\newcommand{\TC}{\mathrm{\check{C}}}
\newcommand{\MFtc}{D_{\TC}}
\newcommand{\Th}{\TC}
\newcommand{\CP}{\TC(P)}
\newcommand{\CO}{\TC(\cO_X)}
\newcommand{\CQ}{\TC(Q)}
\newcommand{\dC}{d_{\scriptscriptstyle \text{\v{C}ech}}}

\newcommand{\Ihkr}{I_{\scriptscriptstyle \text{HKR}}}

\newcommand{\uMFtc}{\underline{D}_{\TC}}
\newcommand{\sP}{\mathscr{P}}
\newcommand{\sAtw}{\sA _{\Th, w}}

\newcommand{\Grad}{\nabla\!\!\!\!\nabla}
\newcommand{\cD}{\mathcal{D}}

\newcommand{\MFGdg}{D _{\vC}^G}
\newcommand{\frs}{\mathfrak{s}}
\newcommand{\frt}{\mathfrak{t}}
\newcommand{\tMF}{\widetilde{D _{\vC}^G}}

\newcommand{\HN}{HN}
\newcommand{\ChHN}{\Ch _{HN}}
\newcommand{\chHN}{\ch _{HN}}

\begin{document}
\title{A chain-level HKR-type map and a Chern character formula}  

\author[K. Chung]{
Kuerak Chung}
\address{Department OF Mathematics Education\\
  Korea National University of Education\\
250 Taeseongtabyeon-ro, Gangnae-myeon\\
Heungdeok-gu, Cheongju-si, Chungbuk 28173\\
 Republic of Korea}
 \email{krchung@kias.re.kr}

\author[B. Kim]{Bumsig Kim}
\address{Korea Institute for Advanced Study\\
85 Hoegiro, Dongdaemun-gu \\
Seoul 02455\\
Republic of Korea}
\email{bumsig@kias.re.kr }

\author[T. Kim]{Taejung Kim }
\address{Department OF Mathematics Education\\
  Korea National University of Education\\
250 Taeseongtabyeon-ro, Gangnae-myeon\\
Heungdeok-gu, Cheongju-si, Chungbuk 28173\\
 Republic of Korea}
\email{tjkim@kias.re.kr}

\date{\today}

\thanks{K. Chung and T. Kim are supported by NRF-2018R1D1A3B07043346. B. Kim is supported by KIAS individual grant MG016404.}

\begin{abstract}
We construct a Hochschild-Kostant-Rosenberg-type quasi-isomorphism for the negative cyclic homology of the category of 
global matrix factorizations on a smooth separated scheme of finite type over a field. 
The map is explicit enough to  yield a negative cyclic Chern character formula
for global matrix factorizations. We also extend these results to the equivariant case of a finite group.
\end{abstract}


\maketitle

\tableofcontents

\section{Introduction}

\subsection{Overview}
Fix a field $k$ of characteristic zero and let $X$ be a smooth separated scheme  of finite type over  $k$.
Let $w$ be a global function on $X$ and let $\GG$ be either the group $\ZZ$ or $\ZZ/2$.
The structure sheaf $\cO _X$ of $k$-algebras  is defined to be concentrated in degree $0$ as a $\GG$-graded sheaf. 
We require that the degree of $w$ is 2. Hence if $w$ is nonzero,
$\GG$ is $\ZZ/2$. If $w=0$, we allow either $\GG$. Assume that the critical values of $w$ are possibly  zero.

We consider a differential $\GG$-graded (in short dg) category $\MFdg (X, w)$ of 
 $\GG$-graded matrix factorizations   for $(X, w)$. By a matrix factorization $(P, \delta _P)$ for $(X, w)$ we mean  
 a locally free coherent $\cO_X$-module $P$ with a $\GG$-grading and $\delta _P$ is a curved differential of $P$.  By definition, $\delta _P$ 
 is a  degree $1$  $\cO_X$-endomorphism of $P$
 satisfying $\delta _P ^2 = w \cdot \mathrm{id}_P$.
 The Hom space from $P$ to $Q$ in $\MFdg (X, w)$  is taken to be a dg $k$-module  $\Hom _{\cO_X}^{\bullet} (I(P), I(Q))$ with 
the chosen quasi-coherent  curved injective replacements $I(P)$, $I(Q)$ of $P$, $Q$ respectively; see \cite{BFK: Kernels}. 
Note that when $\GG = \ZZ$, then $w=0$ and 
the homotopy category of $\MFdg (X, w)$ is triangulated and  equivalent to the derived category 
of bounded complexes of coherent sheaves on $X$.
When $w\ne 0$, then $\GG = \ZZ/2$ 
and the homotopy category of $\MFdg (X, w)$ is the derived category 
of matrix factorizations on $X$. If  $w\ne 0$ in every component of $X$, the homotopy category is also equivalent to the singularity category
of the hypersurface $w^{-1}(0)$; see \cite[Theorem~2]{Orlov: nonaffine} and \cite[Proposition 2.13]{LP}.

The periodic cyclic homology and the Hochschild homology 
of $\MFdg (X, w)$ play the roles of the de Rham cohomology and the Hodge cohomology 
of the ``space" $(X, w)$, respectively. 
The negative cyclic homology $\HN_*(\MFdg (X, w))$ encapsulates both homology theories, which is canonically isomorphic to
\[ \HH^{-*} (X, (\Om _{X/k} [[u]], - dw + ud )); \] 
see \cite[Theorem 1.4]{Ef}. 
Here $u$ is a formal variable with degree $2$ and by definition the exterior derivative $d$ has degree $-1$.  
Each $(P, \delta _P)$ defines  a tautological class in $\HN_0(\MFdg (X, w))$,  the 
so-called categorical Chern character $\ChHN(P)$.
Under the isomorphism, $\ChHN(P)$  can be 
written as an element $\chHN (P)$ of $\HH^{0} (X, (\Om _{X/k} [[u]],  - dw + ud))$. In this paper we find a formula for $\chHN (P)$ by Chern-Weil theory and a \v{C}ech
resolution,  which is a globalization of Brown -- Walker's formula \cite{BW}.

\subsection{A Chern character formula via Chern-Weil theory}

We fix a finite open affine covering $\fU: = \{ U_i \}_{i \in I}$ of $X$ and a total ordering of the index set $I$.
Let $(P, \delta _P)$ be a matrix factorization for $(X, w)$ as before.
Since  each homogeneous part of $P|_{U_i}$ corresponds to a 
finitely generated projective module over a noetherian ring $\Gamma (U_i, \cO_X)$, 
$P|_{U_i}$ always admits  a  connection 
\begin{equation} \label{eqn: nabla i} \nabla _{i} : P|_{U_i} \to P|_{U_i} \ot_{\cO_{U_i}} \Omega ^1_{U_i/k} . \end{equation}
It is a $k$-linear degree $-1$ map satisfying $\nabla _i (p a) = \nabla _i (p) a + (-1)^{|p|} p\ot da $ for $p \in P|_{U_i}$,  $a \in \cO _{U_i}$.
We define the associated total curvature
\begin{equation}\label{eqn: formula R}  R :=    \prod_{i\in I} ( u\nabla_i ^2   + [\nabla_i, \delta _P])  + \prod _{i < j, i, j \in I } (\nabla _i - \nabla _j ) \end{equation} 
as an element in the \v{C}ech complex $\vC (\fU, \cE nd (P) \ot _{\cO_X} \Omega _{X/k}^{\bullet} [[u]] ) $; see \eqref{eqn: total curvature} and 
Remark \ref{rmk: some exp}.
Note that the total degree of each term in the curvature $R$  is zero.

\begin{Thm}\label{thm: ch cech form} {\em (Theorem~\ref{thm: ch cech form body})}
The following Chern character formula holds:
\begin{equation}\label{eqn: ch cech form} \chHN (P) = \tr  \exp  (-R)    \end{equation}
in the ordered \v{C}ech cohomology of complexes of sheaves \[\check{H}^0 (\fU,  ( \Omega ^{\bullet}_{X/k}  [[u]]  ,  -  dw + ud )) . \]
Here the expression 
\[ \exp (-R) = \mathrm{id_P}  - R + R^2/2 - \cdots \cdots + (-1)^{\dim X} R^{\dim X}/(\dim X)! \]  uses the product formula  \eqref{eqn: ACW}  in the \v{C}ech cochain complex  and
$\tr$ denotes the supertrace \eqref{eqn: supertrace for P}.
\end{Thm}

When $|I|=1$ so that $X$ is affine, 
formula~\eqref{eqn: ch cech form} is the main result of  Brown and Walker \cite{BW}.
When $u=0$, 
$\chHN(P)$ becomes the image $\ch _{HH} (P)$ of  the Hochschild homology valued Chern character $\Ch _{HH}(P)$. 
Formula~\eqref{eqn: ch cech form} specialized to $u=0$ 
matches with those of C\u{a}ld\u{a}raru \cite{Cal} for $\GG = \ZZ$; and Platt  \cite{Platt} 
and Kim and Polishchuk \cite{KP} for $\GG = \ZZ/2$, respectively (up to a sign convention which occurs in the choices of
isomorphisms). We will also get a localized Chern character formula; see \S~\ref{sec: local}

\subsection{A chain-level HKR-type map}

We will prove Theorem~\ref{thm: ch cech form} by an explicit quasi-isomorphism, which  is a globalization of the  generalized Hochschild-Kostant-Rosenberg (in short HKR) map in the 
affine case  \cite[Theorem 5.19]{BW}.
The ordinary HKR map \cite[\S~1.3.15]{Loday} uses the canonical de Rham differential $d$ of $\cO _X$. 
While $P$ does not admit a global connection in general, there is a connection  on $P|_{U_i}$. 
This leads us to consider the \v{C}ech model $\MFtc (X, w)$ of $\MFdg (X, w)$; see \S \ref{def: MFtc}.
Let $\CP$ denote the sheafified version of the ordered \v{C}ech $k$-complex  $\vC (\fU, P)$ for $P$; see \eqref{def: sheaf cech}. 
 The objects of $\MFtc (X, w)$ are those of $\MFdg (X, w)$. The Hom space $\Hom (P , Q)$  is
defined to be the dg $k$-space
\begin{equation}  \label{eqn: Hom} \begin{aligned} \Hom _{\CO}^{\bullet} (\CP, \CQ) & \cong \Gamma (X,   \cH om _{\cO _X}^{\bullet} (P,  Q)  \ot _{\cO_X} \CO  ) \\
                      & = \bigoplus _{p} \vC ^ p (\fU, \cH om _{\cO _X}^{\bullet} (P,  Q) )   . \end{aligned} \end{equation}
Its underlying graded $k$-space is  the linear space of graded right $\CO $-linear sheaf homomorphisms from $\CP$ to $\CQ$.
Its differential is given as follows. 
Let $\delta _{\CP} = \delta _P + \dC$, the sum of the differential $\delta _P$ of $P$ and  the \v{C}ech differential $\dC$. 
Then the differential of $\Hom (P, Q)$ is defined by sending 
$ f \in   \Hom (P, Q)  $ to \[ \delta _{\CQ} \circ f - (-1)^{|f|} f \circ \delta _{\CP}  \in  \Hom (P, Q)  , \]
 where $|f|$ denotes the degree of $f$.

For each matrix factorization $P$ for $(X, w)$ and $i\in I$, 
choose, once and for all,  a connection $\nabla _{P, i} $ on $P|_{U_i}$. Hence we have a zero-th global \vC ech element 
\[ \nabla _{P} := \prod _{i \in I} \nabla _{P|_{U_i}}  \]
of the $k$-sheaf $\cH om _{k} (\vC (P), \vC (P) \ot \Omega ^1_X )$.

When $P=\cO_X$, we choose  the de Rham differential as a connection on $\cO _{U_i}$. 
Let $\Hoch (\MFtc (X, w) )$ denote the Hochschild complex
of  the dg category $\MFtc (X, w)$. Using the supertrace $\tr$ \eqref{eqn: supertrace for P} and the product \eqref{eqn: ACW}
we  define a $k[[u]]$-linear map 
\begin{equation}\label{int: def tr nabla}   \tr _{\nabla } : \Hoch (\MFtc (X, w) )[[u]] \to \vC (\fU, \Om _{X/k} )  [[u]]   \end{equation}
by sending
\begin{equation} \label{eqn: tr formula}
 \ka _0 [ \ka _1 | ... | \ka _n]    \mapsto
\sum_{(j_0, ..., j_n)} (-1)^{J}   \tr\big( \frac{ \ka _0 R_1^ { j_0 } [\nabla , \ka _1] R_2^{j_{1}} [\nabla, \ka _2]  ...  R_n^{j_{n-1}}  [\nabla  , \ka _n] R_0^{j_n}  }{(n+J )!} \big)
 \end{equation} for  $ \ka _j \in \vC  (\fU, \cH om_{\cO _X} (P_{j+1} ,  P_{j} ) ) $ with $P_{n+1} =P_0$,
 where $j_i \in \ZZ_{\ge 0}$ and $J= \sum _{i=0}^n j_i$; and
\begin{align*} [\nabla , \ka _j] & :=     \nabla _{P_j } \cdot   \ka_j -   (-1)^{|\ka _j|  } \ka_j \cdot  \nabla _{P_{j+1} }  ;   \\ 
R_j & :=  u \nabla _{P_j} ^2  +  [ \nabla _{P _j } ,  \delta _{P _j} ]  + \prod _{i_0 < i_1} ( \nabla _{P _j , i_0} -   \nabla _{P _j , i_1}) ;
\end{align*} see \eqref{eqn: concrete form} for unwinded expressions.
Note that when $|I|=1$ and $P_i$ is $\cO _X$ for $i=0, ..., n$,  formula \eqref{eqn: tr formula} is  the usual HKR map.

\begin{Thm} \label{thm: tr nabla} {\em (Theorem~\ref{thm: tr nabla body})}
Let $b$ denote the sum of Hochschild differentials from compositions and differentials; and let $B$ be Connes' boundary map.
Then  the $k[[u]]$-linear map
\begin{multline}\label{eqn: thm tr nabla}
  \tr _{\nabla} :  \ \   (\Hoch (\MFtc (X, w) )[[u]], b + uB  )   \\ 
\to   (\vC (\fU, \Omega ^{\bullet}_{X/k}  )  [[u]] ,   \dC - dw + ud )
 \end{multline}
is a quasi-isomorphism  compatible with the HKR-type isomorphism. 
\end{Thm}

Again when $|I|=1$ (i.e., the affine case), this is a theorem of Brown and Walker \cite{BW}, which generalizes a work of
Segal \cite{Se}.

\subsection{The equivariant case by a finite group}
As an application of Theorem~\ref{thm: tr nabla} we can get an equivariant Chern character formula:
Let $X$ come with a finite group $G$ action and assume that $w$ is $G$-fixed. Furthermore, assume that $X$ is quasi-projective.
We consider the dg category $\MFGdg (X, w)$ of $G$-equivariant matrix factorizations for $(X, w)$. 
We choose a finite open affine covering $\fU = \{ U_i\}_{i\in I}$ such that each $U_i$ is $G$-invariant. 
For a finite open affine covering of the $g$-fixed locus $X^g$, let us use $\{ X^g \cap U_i \} _{i \in I}$. 
Let $\nabla _{P|_{U_i}} $ be a $G$-equivariant connection  on  $P|_{U_i}$. Define a connection $\nabla _{P|_{X^g}, i} $ on $P|_{X^g \cap U_i}$ 
to be the restriction of $\nabla _{P|_{U_i}} $ to $X^g \cap U_i$ 
followed by the natural map \[P|_{X^g \cap U_i} \ot _{\cO _{X^g\cap U_i}}  \Omega ^1_{U_i/k}|_{X^g \cap U_i} \to P|_{X^g \cap U_i}\ot_{\cO _{X^g\cap U_i}} \Omega ^1_{X^g \cap U_i/k} \ . \]

\begin{Thm}\label{thm: ch orb}  Let $w_g := w|_{X^g}$.
\begin{enumerate}
\item \label{item one} {\em (Theorem~\ref{thm: ch orb body})}
 The mixed Hochschild  complex $(\Hoch (\MFGdg (X, w)), b, B)$ is naturally quasi-isomorphic to the coinvariant mixed complex
\[   \left(\oplus _{g \in G}  (\Gamma ( X^g,  \vC(\Omega ^{\bullet}_{X^g /k })), \dC - dw_g , d ) \right)_G . \]

\item \label{item two} {\em (Formula \eqref{eqn: equiv HH ch form})} For $P\in \MFGdg (X, w)$, 
 its $G$-equivariant Chern character $\ch ^G_{HH} (P)$  becomes
\begin{equation*}
\ch ^G_{HH} (P) = \frac{1}{|G|} \bigoplus _{g \in G}  \tr  \left( \varphi _{g|_{X^g}} \exp (  - [ \prod _{i} \nabla _{P|_{X^g}, i }   , \delta _{P|_{X^g}}  + \dC ] ) \right) , \end{equation*}
\end{enumerate}
where $\varphi_g  \in \End_{\cO _{X^g}} (P|_{X^g} ) $ is  defined by the multiplication by $g$.
\end{Thm}

Specializing to the case where X is affine space, this recovers the results
\cite[Theorem~2.5.4 \& Theorem~3.3.3]{PV: HRR} of
Polishchuk and Vaintrob. When $w=0$ and $\GG = \ZZ$, Theorem~\ref{thm: ch orb}~\eqref{item one}  is the main theorem of Baranovsky \cite{Bar}. 
In fact we will establish Theorem~\ref{thm: ch orb} combining Theorem~\ref{thm: tr nabla} with his works in \cite{Bar}.
Somewhere else we will treat the case of smooth separated DM stacks of finite type over $k$.

 \subsection{Outline of the paper}  In \S~\ref{sec: hth} 
 we collect definitions and fundamental facts of curved dg $k$-categories; and their modules, mixed Hochschild  complexes, and negative cyclic complexes.
In \S~\ref{sec: ccc}  we recall the definitions of categorical Chern characters and find their alternative expressions; see Proposition \ref{prop: eta}.
In \S~\ref{sec: gmf} we discuss the \v{C}ech model of the dg category of global matrix factorizations and introduce the sheafifications of mixed Hochschild  complexes.
We also collect various invariance results which we will use later.
In \S~\ref{sec: c ecm} we construct connections on a \v{C}ech resolution of global matrix factorizations. Using those connections we define a cochain map $\tr _{\nabla}$ 
and prove Theorem~\ref{thm:  tr nabla}.
In \S~\ref{sec: app} we prove Theorems~\ref{thm: ch cech form} and~\ref{thm: ch orb}.

 \subsection{Acknowledgements} 
 K.C. and T.K. thank Chanyoung Sung and Hoil Kim for encouragement and support.
The  authors thank Michael Brown,  Jaeyoo Choy,  David Favero, and Ed Segal for careful reading and suggestions.

 \subsection{Conventions}
 Unless otherwise stated, ``graded" means $\GG$-graded.
 For a (graded) sheaf $\cF$ on a topological space and an open covering $\fU = \{ U_i \}_{i \in I}$,
 let $U_{i_0 ... i_p} = U_{i_0} \cap ... \cap U_{i_p}$ and let 
 $\vC ^p (\fU , \cF) = \prod _{i_0 < ... < i_p } \cF (U_{i_0 ... i_p})$. 
 For a $k$-category $\cA$, by $x \in \cA$ we mean an object of $\cA$. For $x, y \in \cA$, by $\cA (x, y)$,  $\Hom (x, y)$ or $\Hom _{\cA} (x, y)$  we mean 
 the Hom space from $x$ to $y$  in $\cA$. 
For an element $a$ of a (bi-)graded $k$-space, the (total) degree of $a$ is denoted by $|a|$. 
The commutator $[\ ,\ ]$ is the graded commutator. Often we write $1_P$ for $\mathrm{id}_P$.
The symbol $\simeq$ indicates a quasi-isomorphism.

\section{Hochschild-type invariants} \label{sec: hth}

In this section we collect definitions of curved dg $k$-categories, their mixed Hochschild  complexes, and their negative cyclic complexes.
The main references of this section are \cite{BW, PP}.

\subsection{Curved dg $k$-categories}\label{sub: dg k cat}
\subsubsection{Curved dg algebras}
Let $k$ be a field of characteristic zero.
A $k$-category is   a $k$-algebra with several objects. 
A $k$-algebra is a $k$-module with a unital associative $k$-linear multiplication.
Let $A$ be a graded algebra over $k$ and let $h$ be a degree $2$ element of  $A$. Let $d_A$ be a degree $1$, $k$-linear map such that 
(1) the Leibniz rule holds (i.e., $d _A (a_1a_2) = d_A (a_1) a_2 + (-1)^{|a_1|} a_1 d_A (a_2)$ for $a_1, a_2 \in A$), 
(2) $d_A ^2 = [h , ?]$ (i.e., $d_A ^2 (a) = h a - a h$ for $a \in A$), and (3) $d_A (h) = 0$. Then we call
 $(A, d_A, h)$ a  {\em curved dg {\rm (\rm in short cdg)} $k$-algebra}. 
 Furthermore, if $h=0$, then we call  $(A, d_A)$ a  {\em dg $k$-algebra}.

\subsubsection{Curved dg categories}
There is the notion of a {\em cdg $k$-category} $\cA$ as a cdg $k$-algebra with several objects.
It is a graded $k$-category with a differential $d_{x, y}$ of $\cA (x, y)$ and a degree $2$ element $h_x \in \cA (x, x)$ for every $x, y \in \cA$ 
satisfying the Leibniz rule, $d_{x, y} ^2 (f) = h_y \circ f - f \circ h_x$ for $f\in \cA (x, y)$, and $d_{x, x} ( h_x ) = 0$. The element $h_x$ is called
the {\em curvature} of $x$. If $h_x=0$ for all $x\in \cA$, we call $\cA$ a {\em dg $k$-category}.

\subsubsection{Precomplexes}
There is a cdg category $\PreCom (k)$,  the category of {\em precomplexes} of $k$-modules:
an object is a  graded $k$-module $C$ with a $k$-linear degree $1$ map $\delta _C$
and the Hom space from $C$ to $D$ has a differential $[\delta _{?}, ]$.
The curvature of $C$ is by definition $\delta _C^2$. The full subcategory $\Com _{dg} (k)$ of all objects with vanishing 
curvatures is called the {\em dg category of $k$-complexes}.

\subsubsection{Cdg functors}\label{sub: cdg fun}
A pair $(F, \ka )$  is called  a {\em quasi-cdg functor} 
from a cdg category $\cA$ to another cdg category $\cB$ if 
$F$ is a $k$-linear degree $0$ homogeneous functor 
and $\ka$ is an assignment of a degree $1$ element $\ka_x \in \cB (Fx, Fx)$ for each $x \in \cA$ such that
 \[F (d_{x, y} (f) ) = d _{ Fx, Fy} (F(f) ) + \ka_y \circ F(f) - (-1)^{|f| } F(f) \circ \ka_x \] for all $x, y \in \cA$ and $f \in \cA  (x, y)$.
Furthermore if   $F (h_x) =  h _{Fx} + d_{Fx,Fx} (\ka_x) + \ka_x ^2 $,  $(F, \ka)$ is called a {\em  cdg functor}. 
We often say simply $F$ is a (quasi-)cdg functor for a (quasi-)cdg functor $(F, \ka)$.
If $\ka_x =0$ for every $x$, we call $F$ {\em strict}.
In particular if $\cA$, $\cB$ are dg categories and $F$ is a strict cdg functor, $F$ is just called a {\em dg functor}.
Informally a dg functor is a set of several \lq\lq cochain" maps.  
If $(G, \beta ) : \cB \to \cC$ is a quasi-cdg functor, then the composition of $(G, \kb)$ and $(F, \ka)$ is defined as $(G\circ F, G \circ \ka + \kb \circ F)$.
Note that the composition is  a quasi-cdg functor and the composition of cdg functors  becomes a cdg functor.

\subsubsection{The opposite category}
The opposite category $\cA^{op}$  of $\cA$ is a cdg category whose objects are 
the objects of $\cA$ and whose morphisms are $\cA ^{op} (x, y) := \cA (y, x)$ with differential $d_{y, x}$.
The composition $g \circ f $ in $\cA ^{op}$  is $(-1)^{|f||g|} f \circ g $ in $\cA$ and
the curvature of $x$ in $\cA ^{op}$ is $-h_x$.

\subsubsection{Left and right (quasi-)cdg modules} 
A left (resp., right) {\em quasi-cdg module} over a cdg category $\cA$ is a strict quasi-cdg functor $F: \cA \ (\text{resp., } \cA ^{op} ) \to \PreCom (k) $.
A left (resp., right) {\em cdg module} over $\cA$ is a strict cdg functor $F : \cA \ (\text{resp., } \cA ^{op} )  \to \PreCom (k)$. 
If $\cA$ is a dg category and $F : \cA  \ (\text{resp., } \cA ^{op} )  \to Com_{dg} (k) $ is a dg functor,  we call $F$ is a left (resp., right) {\em dg module} over $\cA$.

The category  $q\Mod_{cdg} (\cA)$ of right quasi-cdg modules over $\cA$ has a natural cdg structure:
The Hom space $\Hom (F, G)$  of two quasi-cdg modules $F, G : \cA ^{op} \to \PreCom (k)$
 is the $\GG$-graded $k$-space of $k$-linear homogenous natural transformations $\mathfrak{n} : F \Rightarrow G$.
The differential $\delta$ of 
$\Hom (F, G)$ is given by $(\delta \mathfrak{n} ) (x) := [ d_{F(x), G(x)}, \mathfrak{n} (x)]$ and the curvature $h_F \in \Hom (F, F)$ of $F$ is given by $h_{F} (x) = h_{F(x)} - F(h_x)$.
The full subcategory $\Mod_{cdg} (\cA)$ of $q\Mod_{cdg} (\cA)$ consisting of right cdg modules over $\cA$ has a natural dg structure.
If $\cA$ is a dg category, we write $\Mod_{dg}(\cA)$ for $\Mod_{cdg} (\cA)$.

\subsubsection{Matrix factorizations}\label{subsub matrix affine}
 By unwinding the definition, a right quasi-cdg module $M$ over $(A, d_A, -h)$ is a right graded $A$-module $M$
with a degree $1$, $k$-linear map $\delta : M \to M$ such that $\delta (m a) = \delta  (m) a + (-1)^{|m|} m d_A (a)$ for every $m\in M$, $a \in A$.
Its curvature is defined to be $\delta ^2 + \rho _{-h} \in \End_A (M)$  where $\rho _{-h}$ is the right multiplication by $-h$.
(1) For example, $(A, d_A)$ can be regarded as a right quasi-cdg module over $(A, d_A, -h)$ with curvature $\lambda _{-h}$.
(2) Exactly when the curvature $\delta ^2 + \rho _{-h}$ of $M$ is zero, $M$ is called a right cdg-module over $(A, d_A, -h)$. Furthermore
we call the right cdg module $(M, \delta)$  over $(A, 0, -h)$ a {\em matrix factorization  for $(A, h)$ or for $(\Spec A, h)$}
when $A$ is a commutative $k$-algebra concentrated in degree $0$,   
and $M$ is  a  finitely generated projective $A$-module. The full subcategory of $\Mod_{cdg}(A, h)$ consisting of matrix factorizations will be denoted by
$\MFdg (A, h)$ (and also by $\MFdg (\Spec A, h)$).

\subsubsection{Quasi-Yoneda}
There is a quasi-Yoneda embedding \[\cA \to q\Mod _{cdg} (\cA) , \ x \mapsto \cA ( \ , x) \] generalizing the usual Yoneda embedding in the case of dg categories.
We define the full cdg-subcategory $q\Perf (\cA)$ of  $q\Mod _{cdg} (\cA)$ as the smallest cdg-subcategory containing  $\cA$ and  closed under  
finite operations of finite direct sum, shift, twist,
 and passage to a direct summand; see \cite{BW, PP} for details.  We call an object of $q\Perf (\cA)$ a 
{\em perfect right quasi-$\cA$-module}.
If its curvature vanishes, then call it a perfect right $\cA$-module. We define $\Perf (\cA)$ as the full dg-subcategory of $\Mod _{cdg} (\cA)$ consisting of 
all perfect right $\cA$-modules.

\subsection{Mixed Hochschild complexes}\label{subsec: MHC}

\subsubsection{The category of mixed complexes} 
Consider a $\GG$-graded dg algebra $k[B]$ defined as a 2-dimensional graded $k$-algebra generated by $B$ with relation $B^2=0$, degree $|B|=-1$, and trivial differential.
A dg $k[B]$-module is called a mixed complex.
For example, for a cdg category $\cA$ we will have a mixed complex $\MC (\cA)$ as defined in \eqref{eqn: MC Hoch} below.
A morphism $\phi: M \to M'$ between mixed complexes is a dg $k[B]$-module homomorphism, i.e., a $k$-linear map preserving degree and both differentials.
This defines the category  $\mathrm{Com} (k[B])$ of mixed complexes.
We call $\phi$ a {\em quasi-isomorphism} if it is so as a cochain map $(M, b) \to (M' , b')$.

Let $u$ be a formal variable with degree $2$. Then we have the induced map $\phi (u) : (M\ot _k k[[u]], b + uB) \to (M' \ot_k k[[u]],  b' + uB)$ in the category of complexes. 

\begin{lemma}\label{lem: Mix to Neg} {\em (\cite[Proposition 2.4]{GJ})} 
If $\phi$ is a quasi-isomorphism, then so is $\phi (u) $.
\end{lemma}

\subsubsection{Hochschild complexes}
For any cdg $k$-category $\cA$, we may consider the Hochschild  complex 
$\Hoch (\cA)$ of $\cA$.
Let $\cA (x, y) [1]$ denote the degree shifted $\cA (x, y)$ by $1$.
 The  canonical degree $-1$ map 
$\cA (x, y) \to \cA (x, y) [1]$ is denoted by $\mathsf{s}$ so that $|\mathsf{s} a  | = |a| - 1$.
Let
\begin{multline*}
 \Hoch (\cA) =  \bigoplus _{x \in \cA} \cA (x, x) \oplus \\
 \bigoplus _{n\ge 1} \big( \bigoplus _{(x_0, ..., x_n) \in \cA ^{\ot n+1} } 
 \cA (x_1, x_0) \ot _k \underbrace{\cA (x_2, x_1) [1] \ot_k ... \ot_k  \cA (x_0, x_n) [1]}_{n}\big)  ,
\end{multline*} 
with differential $b:= b_2 + b_1 + b_0$ defined as follows. 
Denote 
\begin{multline} \label{def: tensor deg} 
a_0[a_1 | ... | a_n] := a_0 \ot  \mathsf{s} a_1 \ot ... \ot \mathsf{s} a_n  \\
 \in   \cA (x_1, x_0) \ot  \cA (x_2, x_1) [1] \ot ... \ot  \cA (x_0, x_n ) [1] .\end{multline} 
Then define
\begin{align*} 
b_2 (a_0[a_1 | ... | a_n]) & :=   (-1)^{|a_0|}  a_0 a_1 [ a_2 | ... | a_n] + (-1)^{|a_0| + |a_1| -1}  a_0 [ a_1 a_2 | a_3| ... | a_n]  \\
                           &  +  \cdots  + (-1)^{\sum_{i=0}^{n-1} |a_i | - (n-1)}  a_0[ a_1 | ... | a_{n-1} a_n]  \\
                           &+  (-1) ^{(|a_n|-1 ) (\sum_{i=0}^{n-1} |a_i| - (n-1))} a_n a_0 [ a_1 | ... | a_{n-1} ] ;\\
b_1   (a_0[a_1 | ... | a_n])  & :=   d (a_0) [a_1 | ... | a_n] + (-1)^{|a_0| -1} a_0 [ d (a_1) | a_2 | ... | a_n] \\
                          & + \cdots   +  (-1)^{\sum_{i=0}^{n-1}  |a_i| - n} a_0 [ a_1| ... | a_{n-1} | d (a_n) ] ; \\                    
  b_0    (a_0[a_1 | ... | a_n]) & :=  (-1)^{|a_0|} a_0 [ h | a_1 ... | a_n]  + \cdots  
                         +   (-1)^{ \sum _{i=0}^n |a_i| - n} a_0 [ a_1 | ... | a_{n} | h ] .
                         \end{align*}
    The homology  \begin{equation*}  HH _* (\cA ) =  H^{-*}  (\Hoch (\cA ), b )  \end{equation*}  
    is called the {\em Hochschild homology} of $\cA$.

  When $\cA$ is a curved dg algebra $(A, d, h)$, we write $C (A, d ,h )$ for $C(\cA)$.

\subsubsection{Connes' boundary map}
On the graded $k$-module $\Hoch (\cA)$, there is  another boundary map, the Connes boundary map, $B = (1-t^{-1})sN$,
where  
\begin{align*} t (a_0[a_1| ... | a_n] ) & :=  (-1)^{(|a_0|-1) \sum_{i=1}^n  (|a_i| -1) } a_1[a_2| ...| a_n| a_0]  ;\\
 s (a_0[a_1| ... | a_n] ) & :=   1[a_0| a_1 | ... | a_n] ; \\
 N (a_0[a_1| ... | a_n] ) & :=   \sum _{i=0} ^n t^i (a_0[a_1| ... | a_n] ) . \end{align*}
Consider
\begin{equation}\label{eqn: MC Hoch} \MC (\cA) := (\Hoch (\cA), b, B ). \end{equation}
Since $bB + Bb = 0$,  $\MC (\cA) $ is a mixed complex, called  the {\em mixed Hochschild complex} of $\cA$.

 It is known that $HH_* (\cA )$ vanishes if the cdg category $\cA$ contains an object with a nonzero curvature; see \cite{CT, PP}.
This motivates the following.

\subsubsection{Hochschild complexes of the second kind} 
For a cdg category $\cA$ one can take the underlying graded $k$-module to be
\begin{multline*}  \Hoch^{II} (\cA) =  \bigoplus _{x \in \cA} \cA (x, x) \oplus \\
\prod _{n\ge 1} \Big( \bigoplus _{(x_0, ..., x_n) \in \cA ^{\ot n+1} } 
 \cA (x_1, x_0) \ot _k \underbrace{\cA (x_2, x_1) [1] \ot_k ... \ot_k  \cA (x_0, x_n) [1]}_{n} \Big) ,
\end{multline*} which  has the corresponding differentials $b_i$, $i=0, 1, 2$. This complex is called the {\em Hochschild complex of the
second kind}.

\subsubsection{Normalized Hochschild  complexes}\label{sub: normalized}
Consider a subcomplex $D$ of $\Hoch (\cA)$ generated by elements $a_0[a_1| ... | a_n]$ for which 
$a_i = t \cdot \mathrm{id}_x$ for  $x \in \cA$, $t\in k$ and some $i\ge 1$. Define 
the normalized Hochschild complex 
\[ \overline{\Hoch} (\cA) =  (\Hoch (\cA), b) / D =: \bigoplus _{n\ge 0}   \overline{\Hoch}_n (\cA), \]
where $n$ denotes the tensor degree as in \eqref{def: tensor deg}.  
The Connes operator descends to an operator on  $\overline{\Hoch} (\cA)$, which will be
also denoted by $B$ by the abuse of notation. We let $\overline{\MC} (\cA) := (\overline{C} (\cA), b, B )$.

\begin{Prop} {\em (\cite[\S 5.3]{Loday} and \cite[\S 3]{Her}.) }
For a dg category $\cA$ the natural map 
\begin{equation}\label{eqn: quot I}  \MC  (\cA) \xrightarrow{\sim} \overline{\MC} (\cA) \end{equation}  is a quasi-isomorphism. 
\end{Prop}

Likewise, there is  the normalized Hochschild complex $\overline{\Hoch}^{II} (\cA)$ of the second kind.
There are corresponding Connes' operators $B$ and mixed complexes
 $\MC ^{II} (\cA) := (C^{II} (\cA), b, B )$, $\overline{\MC} ^{II} (\cA) := (\overline{C}^{II} (\cA), b, B )$.

\begin{Prop} {\em (\cite[Proposition~3.15]{BW}) }  For a cdg category $\cA$ the quotient map 
 \begin{equation}\label{eqn: BW 3.9}      quot^{II}:  \MC ^{II} (\cA) \xrightarrow{\sim} \overline{\MC}^{II} (\cA) \end{equation}     is a quasi-isomorphism. 
\end{Prop}

\subsubsection{Negative cyclic complexes}
Let $u$ be a formal variable with degree $2$.
Then we can make a  total complex 
\[ (\Hoch (\cA )[[u]] : = \Hoch (\cA) \ot _k k[[u]] , b + uB)  \]  It  is called the  negative cyclic complex  of $\cA$ 
and its homology
 \begin{equation*}  HN _* (\cA ) =  H^{-*}  (\Hoch (\cA )[[u]], b + uB)  \end{equation*}
 is called the {\em negative cyclic homology} of $\cA$.
Similarly, one may define   its variants 
  \begin{equation*}   \overline{HN} _* (\cA )  , 
  HN _*^{II}  (\cA ) ,  \overline{HN}^{II} _* (\cA )  .  \end{equation*}

\subsubsection{Functoriality}
Let $\text{cdg-cat}_k$ denote the category of small cdg $k$-categories. The morphisms between them are cdg functors. 
We define a functor 
  \begin{equation*}  \text{cdg-cat}_k \xrightarrow{\overline{\MC}^{II}} \mathrm{Com} (k[B]) ; \cA , (F, \ka) \mapsto (\oHoch ^{II} (\cA), b, B),  (F, \ka)_*  \end{equation*}
by letting  \begin{multline*} (F, \ka )_*  (a_0[a_1| ... | a_n] ) 
 :=  \sum_{n=0}^{\infty} \sum _{(j_0, ..., j_n) \in \ZZ_{\ge 0} ^{n+1}}  (-1)^{j_0 + ... + j_n}  \\  F(a_0) [\underbrace{\ka _0 | ... | \ka _0}_{j_0} |  
F(a_1) | \underbrace{\ka _1  |...| \ka_1}_{j_1} | F(a_2) | ... | F(a_n) | \underbrace{\ka _n  | ... | \ka _n}_{j_n}  ]     .  \end{multline*}
Here $\ka _i : = \ka _{\text{domain of } F (a_{i})} $. This is indeed a functor; see \cite{BW, PP}. 

\begin{Rmk}
For a cdg functor $(F, \ka)$, it induces a cochain map $(\Hoch ^{II} (\cA), b) \to (\oHoch ^{II} (\cA '), b')$ but not necessarily a morphism 
of mixed complexes  $\MC ^{II} (\cA) \to \MC^{II} (\cA ')$; see \cite[Remark 3.21]{BW}.
\end{Rmk}

Let $\text{cdg-cat}^{st}_k$ denote  the category of small cdg $k$-categories whose morphisms are {\em strict} cdg functors.
Let $\text{dg-cat}_k$ be the full subcategory of $\text{cdg-cat}^{st}_k$ whose objects are dg categories.
In the same way we have natural functors  
\begin{align*} \text{cdg-cat}^{st}_k  & \xrightarrow{\overline{\MC} } \mathrm{Com} (k[B]) ; \cA , F \mapsto (\oHoch   (\cA), b, B),  F_* \\
 \text{dg-cat}_k &  \xrightarrow{\MC} \mathrm{Com} (k[B]) ; \cA , F \mapsto (\Hoch  (\cA), b, B),  F_* \end{align*}
where $F_*:= (F, 0)_*$.

\subsection{Some invariances}\label{sub: inv}

\subsubsection{Morita invariances}\label{sub: mor}
Let $F: \cA \to \cB$ be a dg functor between dg categories $\cA$, $\cB$. We say that it is {\em cohomologically full and faithful}
if  for every $x, y \in \cA$, the cochain map $F: \cA (x, y) \to \cB (Fx, Fy)$ is a quasi-isomorphism.
Furthermore, if $\cB$ is split-generated by the image of $F$, then $F$ induces a Morita equivalence between $\cA$ and $\cB$.
Thus the morphism 
\begin{equation}\label{eqn: mor}  F_* : \MC (\cA) \xrightarrow{\sim} \MC (\cB) \end{equation} 
is a quasi-isomorphism; see  \cite[Lemma 4.12]{She} and references therein.
The Yoneda embedding $\cA \to \Perf \cA$ is such an example. 
Furthermore if $F$ is essentially surjective, $F$ is called a {\em quasi-equivalence}. 

\subsubsection{Invariance under Pseudo-equivalences}\label{subsub inv pseudo}
Let $(F, \ka): \cA \to \cB$ be a cdg functor between cdg categories $\cA$, $\cB$. We say that $(F, \ka)$ is a {\em pseudo-equivalence}
if $\cA (x, y) \to \cB (Fx, Fy)$ is an isomorphism for all $x, y \in \cA$ and every object of $\cB$ can be constructed from the image of $F$
by a finite sequence of operations of finite direct sum, passage to a direct summand, twist, and shift.

\begin{Thm} {\em (Polishchuk-Positselski \cite{PP})}
For a pseudo-equivalence $(F, \ka)$
\begin{equation}\label{eqn: inv pseudo} (F, \ka)_* : \overline{\MC} ^{II} (\cA) \xrightarrow{\sim} \overline{\MC} ^{II} (\cB) \end{equation}
 is a quasi-isomorphism.
 \end{Thm}
% ;  see \cite{BW, PP}.
The embeddings $\cA \to q\Perf \cA$ and $\Perf \cA \to q\Perf \cA$ are pseudo-equivalences; see \cite{PP}.

\subsubsection{Comparison}\label{subsub comp I II}
Let $X$ be an affine $k$-scheme and let $w : X \to \mathbb{A}^1_k$ be a regular function.
Then we consider the dg category $\MFdg  (X, w)$ of matrix factorizations for $(X, w)$; see \S~\ref{subsub matrix affine}.  

\begin{Prop}\label{prop: I to II}  {\em (\cite[\S 4.8 Corollary A]{PP}) }  If $X$ is a smooth affine scheme finite type over $k$ and $w$ has no other critical values but zero, 
then the natural map 
\begin{equation}\label{eqn: I to II}  C (\MFdg (X, w) ) \xrightarrow{\sim} C^{II}(\MFdg (X, w) )  \end{equation}
is  a quasi-isomorphism. 
\end{Prop}
\begin{proof} 
Without loss of generality we may assume that $X$ is connected. 
When $w\ne 0$, this follows from  \cite[Corollary A]{PP}, since $k$ is a perfect field. 
When $w=0$, this follows from considering the tensor grading of  $C^{II} (\Gamma (X, \cO _X))$ and the HKR isomorphism of \cite{HKR}.
\end{proof}

\section{Categorical Chern characters}\label{sec: ccc}

In this section  we recall the definition of categorical Chern characters
and discuss alternative expressions of them which we will use in
 \S~\ref{sub: orb} for Chern character formulae.

\subsection{Chern characters $\Ch_{HH} (P)$ and  $\Ch _{HN} (P)$} 
For every object $P$ of a dg category $\cA$, note that
the identity map $1_P$ is a degree $0$ cocycle of the Hochschild complex of $\cA$. Hence each $P$ determines the class $[1_P]$ in $HH_0 (\cA)$.
We call $[1_P]$ the Hochschild homology valued Chern character  of $P$, denoted  by $\Ch_{HH} (P)$.
Likewise, since $1_P$ is a cocycle of the normalized negative cyclic complex of $\cA$, we simply define
the negative cyclic homology valued Chern character $\Ch_{HN} (P)$ of $P$  to be the class $[1_P]$ in $\overline{HN}_0 (\cA) \cong HN_0 (\cA)$.

For a cdg category $\cA$ and $P \in q\Perf \cA$ we define Chern characters $\Ch^{II}_{HH} (P)$, $\Ch^{II}_{HN} (P)$ of the second kind in the same manner
using $1_P$ in the normalized complexes of the second kind.

\subsection{Alternative expression via direct summands}\label{sub: alt}

Let $P$ and $N$ be objects of a dg category $\cA$.
Suppose that $P$ is a direct summand of $N$, i.e., there are $g : P \to N$ 
and $f: N \to P$ degree $0$ closed homomorphisms such that $f \circ g = 1_P$.
Note that $1_P$ and the idempotent $\pi : = g \circ f$ are homologous
\begin{equation}\label{eqn: 1  pi} 1_P \sim \pi \end{equation}
 in the Hochschild complex of $\cA$; see \cite{Shk: HRR}.
If $N$ is simpler than $P$, then $\pi$ is often easier to handle than $1_P$.

\subsubsection{}\label{sub: semi hom 1}
More generally, every element of $C (\End _{\cA} P)$ is homologous to an element of $C(\End_{\cA} N)$ in the Hochschild complex of $\cA$:
Consider the inclusion functor $inc$ from  $\{N\}$ to $\{ P, N\}$ between the full dg subcategories of
$\cA $ consisting only one indicated object, two indicated objects respectively.
As a semifunctor (i.e., a dg functor not necessarily preserving identities between dg categories), $inc$ has a left inverse $F$.
\[ \{ N \} \xleftarrow{F} \{ P, N \} \xleftarrow{inc} \{ N \} .  \]
The functor $F : \{ P, N \} \to \{ N \}$ is determined by sending 
$\ka \mapsto  g \circ \ka \circ f $ for $\ka \in \End P$, $\ka ' \mapsto g \circ \ka '$ for $\ka ' \in \Hom (N, P)$,
$\ka '' \mapsto \ka '' \circ f $ for $\ka '' \in \Hom (P, N)$, and $\beta \mapsto \beta$ for $\beta \in \End (N)$.

\begin{lemma}\label{lem: PN} 
In the Hochschild complex $( \Hoch \{ P, N \}, b) $, 
every element $\ka _0 [ \ka _1 | .... | \ka _n]$ is  homologous to  $F(\ka_0 ) [ F(\ka_1 ) | ... | F(\ka_n) ] $.
\end{lemma}
\begin{proof}
The semifunctor $F$ induces a cochain map  $F_* :  ( \Hoch \{ P, N \}, b)   \to (\Hoch \{ N \}, b) $,
which is a left inverse of the quasi-isomorphism $inc_* :   (\Hoch \{ N \}, b)  \to (\Hoch \{ P, N \} , b)$
induced from $inc$. Therefore for every $m\in \Hoch \{ P, N \}$, $m$ is homologous to $inc _* (F_*(m))$, since $F_*(m) = F_* (inc_* (F_* (m)))$.
\end{proof}

\begin{Rmk}\label{rmk: a, ta} 
We can generalize Lemma~\ref{lem: PN}: Consider two full dg subcategories $\cC$ and $\cD$ of $\Mod _{cdg} \cA$ such that $\cC$ is a subcategory of $\cD$ 
and  any object of $\cD$ is a direct summand of an object of $\cC$. Then we can construct a semifunctor from
$\cD \to \cC$ which is a left inverse of the inclusion $\cC \to \cD$ and hence a left inverse of the quasi-isomorphism
$(\Hoch (\cC ), b) \to (\Hoch (\cD), b)$. 
\end{Rmk}

\subsubsection{}\label{subsub: eta pi}
There is a negative cyclic homology version of \eqref{eqn: 1 pi}:
Consider a cycle  \[ \gamma _P :=   1_P +   \sum _{i=1}^{\infty}        (-1)^{i}  \frac{(2i)!}{2(i!)}  2 \cdot 1_P  [\underbrace{1_P |  \cdots | 1_P }_{2i} ] u^{i}  \]
in $(\Hoch \{ P \}  [[u]] , b + uB)$. It projects to $1_P$ in $(\overline{\Hoch}\{ P \}  [[u]] , b + uB)$.
 Let 
 \begin{align}\label{eqn: eta pi}  \eta _{\pi} & :=      \pi +   \sum _{i=1}^{\infty}        (-1)^{i}  \frac{(2i)!}{2(i!)}  (2\pi -1_N)  [\underbrace{\pi|  \cdots | \pi}_{2i} ] u^{i}  \\  \nonumber
   & =  \pi - (2\pi - 1_N) [ \pi | \pi] u + 6(2 \pi - 1_N)[ \pi| \pi | \pi | \pi]  u^2 + \cdots . \\ \nonumber
 \end{align}
It is straightforward to check that $\eta _{\pi}$ is also a cycle in $(\overline{C} \{ N \} [[u]], b + uB)$.

\begin{Prop}\label{prop: eta}  {\em (Proposition \ref{prop: eta appendix})}The two cycles $1_P$ and $\eta _{\pi}$ are  homologous in $(\overline{C} \{ P, N \} [[u]], b+ uB)$.
\end{Prop} 

Proposition~\ref{prop: eta} might be known to the experts. Due to lack of a suitable reference, we will give a proof in  Appendix~\ref{sec: pf 1}. 
Another proof will be given in \cite{CKK2}.

\begin{Rmk}
A combination of Proposition~\ref{prop: eta} and \cite[Theorem~5.7]{BW} immediately 
provides an answer to a question in \cite[Remark~5.23] {BW}.
\end{Rmk}

\section{Global matrix factorizations}\label{sec: gmf}
Let $X$ be a smooth separated scheme of finite type over $k$.
In this section we introduce the notion of global matrix factorizations as certain cdg modules over the sheaf $(\cO _X, 0, -w)$ of cdg algebras.
And we define the \v{C}ech model for a dg enhancement of the derived category of global matrix factorizations. 
We  introduce the sheafifications of mixed Hochschild  complexes that are defined in \S~\ref{sec: hth}. All of these will be used in \S~\ref{sec: c ecm}.

\subsection{Matrix factorizations and  injective models}
We may  sheafify the notion of cdg algebras and (quasi-)cdg modules.
For example, we have \begin{equation}\label{eqn: sA} \sA_w:= (\cO _X, 0 , -w) \end{equation} a sheaf of cdg $k$-algebras: $U\mapsto (\cO _X(U), 0, -w|_U)$ for open subsets $U$ of $X$.
A {\em global matrix factorization}, or simply {\em matrix factorization} for $(X, w)$ amounts to a cdg $\sA_w$-module $(P, \delta _P)$
such that $P$ is a locally free and coherent as an $\cO_X$-module. 
There is a notion of the derived category of matrix factorizations for $(X, w)$; see \cite{Orlov: nonaffine}. It is equivalent to
the homotopy category of $\MFdg (X, w)$ if we take $\MFdg (X, w)$ to be the dg category of
matrix factorizations whose  
 objects are matrix factorizations and whose Hom from $P$ to $Q$ is defined to be $\Hom_{\cO _X} (I_P, I_Q)$ 
 using quasi-coherent curved injective replacements $I_P$, $I_Q$ of  $P$, $Q$, respectively.

\subsection{\v{C}ech models}\label{sub: Cech model}
We will consider another model using \v{C}ech resolutions of matrix factorizations instead of
injective resolutions. 
We fix a finite open affine covering $\fU = \{ U _i \}_{i \in I}$ of $X$. Fix a total ordering of $I$. For an $\cO_X$-sheaf $\cF$, 
 let $ \TC (\cF)  $ the sheafified version of the (ordered) \v{C}ech complex of $\cF$ with respect to the covering $\fU$:
  \begin{equation}\label{def: sheaf cech} \TC ^p (\cF) : = \prod _{i_0 < ... < i_p } f_* (\cF |_{U_{i_0, ..., i_p}}) =\cF \ot _{\cO_X}   \prod _{i_0 < ... < i_p} f_* \cO _{U_{i_0, ..., i_p}}  \end{equation}  
  where  $f$ denotes the immersions $U_{i_0, ..., i_p} : = U_{i_0} \cap ... \cap U_{i_p} \to X$. Here the second equality follows from the projection formula.
   The {\em \vC ech differential} $1_\cF \ot \dC $ for 
   \begin{equation}\label{eqn: Cech tensor order}  \vC (\cF) = \cF \ot  \vC (\cO _X) \end{equation} will be written simply  $\dC$ by abuse of notation.
 We follow the Koszul sign rule so that  $\dC (x \ot a ) = (-1)^{|x|} x \ot \dC (a )$ for $x\ot a \in  \cF \ot  \vC (\cO _X) $.

\subsubsection{ Alexander-\v{C}ech-Whitney products}\label{ACW}
The Alexander-\v{C}ech-Whitney product (see \cite{Swan}) is an $\cO _X$-homomorphism 
\[ \cdot :  \vC (\cO _X) \ot _{\cO _X} \vC (\cO _X) \to \vC (\cO _X) \] defined by setting
\begin{align}\label{eqn: ACW O} (a \cdot b )_{i_0, ..., i_{p+q}} 
 & :=   a_{i_0,..., i_p} |_{U_{i_0,...,i_{p+q}} } b_{i_{p}, ..., i_{p+q}} |_{U_{i_0,...,i_{p+q}}}  \end{align} 
 for  $a \in  \vC^p ( \cO_X ) , b \in \vC^q ( \cO_X)$.
 
 If $\cF$ is a sheaf of graded $\cO _X$-modules, we have the {\em right} $\CO$-module structure on the \v{C}ech resolution 
 $ \vC (\cF) =  \cF \ot \vC (\cO _X) $.
If $\cF$ is a sheaf of graded $\cO_X$-algebras, we define the product structure  on 
 $ \vC (\cF) =  \cF \ot \vC (\cO _X) $ 
by  the formula 
\begin{align}\label{eqn: ACW}  (x\ot a) \cdot  (y\ot b) := (-1)^{p | y |} x y \ot a \cdot b \end{align}  
 for $x,  y \in \cF$  $a \in  \vC^p ( \cO_X ) , b \in \vC^q ( \cO_X)$, and the degree $|y|$ of $y$.

 Note that $\dC  (\alpha\cdot \beta) = (\dC \alpha ) \cdot \beta+ (-1)^{|\alpha |} \alpha\cdot (\dC \beta)$
 for $\alpha, \beta\in \vC (\cF)$.
 Therefore \begin{equation}\label{eqn: sAt} \sAtw :=(\TC (\cO _X), \dC, -w) \end{equation} is a sheaf of cdg $k$-algebras on $X$ as $[-w, ] = \dC ^2$.
 
 \subsubsection{The category $\MFtc (X, w)$}\label{def: MFtc}
For each $P \in \MFdg (X, w)$, consider $\TC (P)$. It comes with
the total curved differential $\delta _P \ot 1 + 1 \ot \dC$ by the identification  $\TC (P) = P \ot_{\cO_X} \TC (\cO_X) $. Abusing notation we simply write 
 \[ \delta _P  + \dC\  \text{ for } \ \delta _P \ot 1 + 1 \ot \dC, \]  but we need to remember the Koszul sign rule when we apply $\dC$. 
 Hence $\TC (P)$ can be regarded as a certain cdg-module over $\sAtw$.
 In the dg category $\Mod _{cdg} (\sAtw) $  of right  cdg-modules over $\sAtw$, let us 
take all the objects of form $(\TC (P), \delta _P + \dC)$, for some $P \in \MFdg (X, w)$. This yields  the full subcategory \[ \MFtc (X, w), \]
which is easily seen to be quasi-equivalent to $\MFdg (X, w)$. We call $\MFtc (X, w)$ the {\em \vC ech model} for $\MFdg (X, w)$ with respect to $\fU$.
We may regard the objects of $\MFdg (X, w)$  as the objects of    $\MFtc (X, w)$.

\subsubsection{The category $q\MFtc (X, w)$}
We will also consider a full subcategory $q\MFtc (X, w)$  of $q\Mod _{cdg} (\sAtw) $  of right  quasi-cdg-modules over $\sAtw$,
consisting of $(\TC (P), \delta _P + \dC)$ such that $(P, \delta _P)$ is a locally free coherent right quasi-cdg module over $\sA _w$.
Here a locally free coherent right quasi-cdg module $(P, \delta _P)$ means that $P$ is a $\GG$-graded locally free coherent $\cO_X$-module and 
$\delta _P$ is an $\cO_X$-linear, degree $1$ map. It is not required that $\delta _P^2 =  -\rho_{-w}$. The curvature element of $(P, \delta _P)$ is
defined to be $\delta _P^2 + \rho _{-w}$.

\subsection{Sheafification}

\subsubsection{} Let $\sP$ be a presheaf of cdg categories on $X$:
\[U \mapsto \sP (U) \in \mathrm{cdg}_{k} \] for each open subset $U$ of $X$.
For example,
$\sA _w$ and $\sAtw$ are sheaves of cdg $k$-algebras on $X$, that is
$U \mapsto (\cO (U), 0, -w|_U)$ and  $U \mapsto (\TC (\cO _X) (U), \dC, - w|_U)$, respectively.
Also, we have a presheaf  \[U \mapsto \MFtc (U, w|_U) \] (resp.,  $q\MFtc (U, w|_U)$) of dg (resp., cdg)  categories, which will be denoted by
$\uMFtc (X, w)$ (resp., $q\uMFtc (X, w)$). Here for each $U$ we use the induced covering $\{ U_i \cap U\} _{ i \in I} $ of $U$ to construct 
$\TC (Q)$ for $Q \in (q)D_{dg} (U, w|_U)$.    Note that for $P \in (q)D_{dg}(X,  w)$, $\TC (P)|_U \cong \TC (P|_U)$.

\subsubsection{} We may sheafify the mixed Hochschild complex of $\sP (U)$:
the sheaf $\uHoch (\sP )$ associated to the presheaf
\[ U \mapsto \Hoch (\sP  (U)).  \] 
  Let $\underline{\MC}(\sP )$ denote this complex of sheaves, i.e., $\uHoch (\sP )$ with the induced differential and Connes' operator, denoted by 
$b$, $B$ by abuse of  notation. Similarly we have the complex $\underline{\overline{\MC}} ^{II} (\sP )$
of sheaves associated to $U \mapsto \oHoch^{II} (\sP (U))$.
For example,  a sheafification of  complexes $U \mapsto \Hoch ( \MFtc (U, w|_U)) $
 is denoted  by  $\uHoch(\uMFtc (X, w))$.

\subsubsection{}  Let $G$ denote the Godement resolution functor
so that  for a sheaf $\cF$ of abelian groups, the canonical flasque resolution of $\cF$ is
denoted by $G(\cF)$. Let $G_{\le d}$ denote the Godement resolution functor canonically truncated at the amplitude $d$ with $d:=\dim X$.
Since $H^i (X, \cF) =0$ for all $i > d$ by a vanishing theorem of Grothendieck (\cite[III. Theorem 2.7]{Hart}), $G_{\le d} (\cF)$ is a $\Gamma$-acyclic resolution of $\cF$.
For a sheaf $(\uC, b, B)$ of mixed (unbounded) complexes on $X$ we define a mixed complex
\begin{equation}\label{eqn: RG mix} \RR\Gamma (\uC, b, B) := ( \Gamma (X, G_{\le d} \uC), G_{\le d} b, G_{\le d} B)  \end{equation} following \cite[Definition 3.22]{Ef}.
The functor $\RR\Gamma$ preserves quasi-isomorphisms.

\bigskip

\noindent {\bf Condition $(\star)$}:  $X$ is a smooth separated scheme finite type over $k$ and $w$ has no other critical values but zero.

\bigskip

 \begin{Prop}\label{prop: Efimov results}  {\em (\cite[Proposition 5.1 \& Proposition 3.24]{Ef})} 

 \begin{enumerate}  

 \item\label{item: Efimov results 1}  Assume $(\star)$ and let $C'$ denote either $C, \overline{C} $  $C^{II}$ or $\overline{C}^{II}$. The natural morphism 
 \begin{equation*}\label{eqn: Efimov results} (\Hoch ' (\MFtc (X, w)) ,  b, B) \to \RR\Gamma (\uC ' , b, B) \end{equation*}
is a quasi-isomorphism between the mixed complexes.

\item\label{item: Efimov results 2} For a sheaf $\cF$ of $k$-vector spaces the natural morphism 
\[ G_{\le d} (\cF \ot k[[ u]] ) \to (G_{\le d} \cF ) \ot k[[u]] \] is a quasi-isomorphism between the complexes of $\Gamma$-acyclic sheaves.

\end{enumerate}
\end{Prop}
\begin{proof}
(1) The  proof of Proposition 5.1 of Efimov \cite{Ef} works for the various $C'$. (2) This is the contents of \cite[Proposition 3.24 \& Lemma 3.25]{Ef}.
\end{proof}

Combining Proposition \ref{prop: I to II} and  Proposition \ref{prop: Efimov results}~\eqref{item: Efimov results 1}
we may remove the affine condition on $X$ in Proposition \ref{prop: I to II}.
\begin{Cor}
  Under the condition $(\star)$ the natural map 
\begin{equation}\label{eqn: I to II general}  C(\MFtc (X, w) ) \xrightarrow{\sim} C^{II}(\MFtc (X, w) )  \end{equation}
is  a quasi-isomorphism. 
\end{Cor}

From now on we will assume the condition $(\star)$ unless otherwise stated.

\subsection{Applications of invariances} \label{sub: cdg invariances}

Note that for each open affine subset $U$ of $X$, we have the following.
\begin{enumerate}

\item\label{item: HKR}  The HKR  map 
\begin{align*} \Ihkr : \overline{\MC}^{II} (\sA_w (U)) & \to   (\Omega ^{\bullet}_{U/k} , d, -dw ) ; \\
 a_0[a_1 | ... | a_n] & \mapsto  
                    \frac{1}{n!} a_0 da_1  ... da_n \end{align*}  is a quasi-isomorphism of mixed complexes; see \cite{CT, Se}. % and Lemma \ref{lem: Hoch to Mix}.

\item\label{item: O to CO}  The natural morphism \[\overline{\MC}^{II} (\sA_w (U)) \to \overline{\MC}^{II} (\sAtw (U)) \]
is a quasi-isomorphism by the spectral sequence argument with the filtration of the complex $(\overline{\Hoch}^{II} (\TC (\cO _X) (U)), b)$ by the \vC ech grading.

\item\label{item: qYoneda} The quasi-Yoneda embedding 
\[ \sAtw (U) \to q\MFtc (U, w|_U) \]
is a pseudo-equivalence and hence induces a quasi-isomorphism in the mixed normalized Hochschild  complexes of the second kind by \eqref{eqn: inv pseudo}.
\item\label{item: Yoneda} The natural embedding 
\[ \MFtc (U, w|_U) \to q\MFtc (U, w|_U) \]
is  a pseudo-equivalence and hence induces a quasi-isomorphism in the  mixed normalized Hochschild  complexes  of the second kind  again by \eqref{eqn: inv pseudo}.
\end{enumerate}

Thus we have quasi-isomorphisms between sheaves of mixed complexes
\begin{align}\label{eqn: qYoneda result}   
\underline{\overline{\MC}}^{II} (\uMFtc (X, w))  & \simeq \underline{\overline{\MC}}^{II} (q\uMFtc (X, w))   &  \text{ by \eqref{item: Yoneda} \ } \\
          \nonumber          &  \simeq    \underline{\overline{\MC}}^{II}  (\sAtw) & \text{ by \eqref{item: qYoneda} \ }       \\
                \nonumber    &  \simeq      (\Om _{X/k}, - dw , d  )                                 &     \text{  by  \eqref{item: O to CO} \& \eqref{item: HKR} }    \\
                 \nonumber   &  \simeq      (\vC(\Om _{X/k}), - dw , d )       .                      &          
\end{align}

\section{Connections and the map $\tr _{\nabla}$}\label{sec: c ecm}

\subsection{$V$-Connections}\label{sub: V conn} 
Let $(A, d_A, -h)$ be a curved dg  $k$-algebra.
Let $V$ be a subalgebra of the $k$-algebra $A$ such that 
\[ d_A |_V =0, \quad h\in V, \]  and $V$ is contained in the even degree part of $A$.
Furthermore assume that 
\[ v a = a v \quad \text{  for every } v\in V, \ a\in A . \] 
Note that $d_A ^2 = [-h, ?] = 0$.
Consider a $k$-linear map \[d: A \to \Omega ^1_{V/k} \ot _V  A \]  satisfying the Leibniz rule
\begin{equation}\label{eqn: derivation rule}  d(a_0a_1) = (da_0) a_1 + (-1)^{|a_0|} a_0 da_1 \end{equation}  for $a_0, a_1\in A$.
Here   if $da_i = dv_i \ot a_i' $ with $v_i\in V, a_i' \in A$, we define
\[ (da_0 )a_1 := dv_0 \ot a_0' a_1 \text{ and } a_0 da_1 := (-1)^{|a_0|} dv_1 \ot a_0 a_1' . \]

On $\Omega ^1_{V/k}\ot _V A $ we define a differential (again denoted by) $d_A$ by extension of scalars:
\[ d_A(dv \ot a) = - dv \ot d_A (a) .\]
Assume that  for every $a \in A$
\begin{equation}\label{eqn: dA d} d_A d(a)  +  d d_A (a) = 0 . \end{equation}

\begin{example}
When $V=A$, then $d$ is the usual $k$-derivation of $V$.
\end{example}

\begin{example}\label{example: V-con} From the de Rham differential $\cO _X \to \Omega ^1_{X/k}$
we have \begin{equation}\label{eqn: de Rham Cech}  d: \CO \to \Omega ^1_{X/k}\ot \CO \end{equation} by the projection formula. 
Let $U$ be an affine open subset of $X$ and let
\[ V = \Gamma (U, \cO _X), A = \Gamma (U, \TC (\cO _X)), d_A = \dC, h = w. \]
Then taking $\Gamma (U, - )$ at \eqref{eqn: de Rham Cech} (and slightly abusing notation) we obtain a canonical $k$-derivation  
\begin{equation}\label{eqn: de Rham d}  d:  \Gamma (U, \TC (\cO _X)) \to  \Omega ^1_{U/k}  \ot _{ \Gamma (U, \cO _X)}  \Gamma (U, \TC (\cO _X))  \end{equation} 
 satisfying \eqref{eqn: derivation rule} and \eqref{eqn: dA d}.
\end{example}

\begin{Def}\label{def: V con}
Let $M$ be a graded $A$-module. We call a $k$-linear degree $-1$ map
\[\nabla : M \to  \Omega ^1_{V/k} \ot _V  M  \]  a {\em $V$-connection with respect to} $d$ above if 
 \begin{equation*} \nabla (m a) = \nabla  (m) a + (-1)^{|m|} m \ot da  \end{equation*}
  for every $a \in A$.
Here $m \ot da \in M \ot _A (A \ot _V \Omega ^1_{V/k}) \cong M \ot _V \Omega ^1_{V/k}$. 
\end{Def}

\begin{example} Note that if $\nabla '$ is another $V$-connection on $M$, then $\nabla - \nabla '$ is an $A$-linear map
$M \to  \Omega ^1_{V/k} \ot _V  M  $ as usual.
If $M = A ^{\oplus n}$, then $M$ has the obvious connection $d_F$ induced from $d$.
So a connection $\nabla$ on $A^{\oplus n}$ can be written as $d_F + C$ for some $n\ti n$ matrix $C$ whose entries are elements of $\Omega ^1_V \ot _V  A $.
\end{example}

\subsection{The map $\tr _{\nabla}$}
Consider a perfect right quasi-cdg module $(P, \delta _P)$ over $(V, 0, -h)$. Note that $\delta _P$ is $V$-linear. 
Let $\cP := P\ot _V A$ and note that $\End _A (\cP) = \End _V (P) \ot _V  A$.
Then we have a cdg algebra \[ (\End _A (\cP), [ \delta _{\cP}, ], \delta _{\cP} ^2 + \rho _{-h}) , \]
where $\delta _{\cP}:= \delta _P\ot 1  + 1 \ot d_A$.
Suppose that $\delta _{\cP}^2 = \rho _ {\nu}$ for some $\nu \in V$.
Here $\rho _{\nu}$ denotes the right multiplication by $\nu$.

Assume that a $V$-connection $\nabla$ on $\cP$ is given and let 
\begin{equation}\label{eqn: abst R} R:=u \nabla  ^2 + [\nabla ,  \delta _{\cP}  ] \ \in \Omega _V^{\bullet}  \ot _V \End _A (\cP ) [[u]] . \end{equation}
Here $\nabla ^2$ denotes $(\mathrm{id} \ot \nabla ) \circ \nabla$ followed by the wedge operation:
\[ \cP \xrightarrow{\nabla} \Omega ^1_{V/k} \ot_V \cP  \xrightarrow{\mathrm{id} \ot \nabla}  \Omega ^1_{V/k} \ot_V ( \Omega ^1_{V/k} \ot_V \cP)
 \xrightarrow{\wedge \ot 1}   \Omega ^2_{V/k} \ot_V \cP.\]
The composition $\nabla ^2$ is $A$-linear. 
The $A$-linearity of  $[\nabla ,  \delta _{\cP}  ]$ requires \eqref{eqn: dA d}.
We regard  $ \Omega ^{\bullet}_{V/k} \ot_V \mathrm{End}_A (\cP )  [[u]] $ as a $k[[u]]$-algebra 
defined by 
\[ ( \gamma_1 \ot \ka_1  ) \cdot ( \gamma _2 \ot \ka_2 ) = (-1) ^{|\ka_1||\gamma _2|} (\gamma _1 \wedge \gamma _2) \ot (\ka_1\circ \ka_2)   \]
  for $\gamma _i \in \Omega ^{\bullet}_{V/k}, \  \ka _i \in  \mathrm{End}_A (\cP )$.

Define a $k[[u]]$-linear map $\tr _{\nabla }$ as the composition 
\begin{align}  \label{eqn: def tr nabla}    
   \Hoch  ( \mathrm{End}_A (\cP )  )[[u]] \to   \Omega ^{\bullet}_{V/k} \ot_V \mathrm{End}_A (\cP )   [[u]]    \xrightarrow{\tr}   \Om _{V/k}[[u]]  \ot _V  A ,  \end{align} 
   where the first map is defined by sending $ \ka _0 [ \ka _1 | ... | \ka _n] $ to     
    \begin{equation}\label{formula tr nabla} 
\sum_{(j_0, ..., j_n)} (-1)^{J}   \tr\big( \frac{ \ka _0 R^ { j_0 } [\nabla , \ka _1] R^{j_{1}} [\nabla, \ka _2]  ...  R^{j_{n-1}}  [\nabla  , \ka _n] R^{j_n}  }{(n+J )!} \big), \end{equation}
 with $j_i \in \ZZ_{\ge 0}$ and $J= \sum _{i=0}^n j_i$; and the second map $\tr$ is the $  \Om _{V/k}[[u]] $-linear extension the supertrace
\begin{equation}\label{eqn: def trace} \mathrm{End}_A (\cP)  \underset{\text{canonically}}{\cong}  (\Hom _A (\cP, A)  \ot _A \cP)   \xrightarrow{\mathrm{ev}}  A  . \end{equation}
Here the canonical isomorphism $\cong$ follows from the fact that $\cP$ is a finitely generated projective $A$-module.
Note that in the sum \eqref{formula tr nabla}, the terms for $j_0 + ... + j_n > \dim X$ vanish.

Similarly we may define the version of $\tr_{\nabla}$ for the negative cyclic complex of the second kind as
\begin{equation}\label{eqn: def tr nabla II} \tr _{\nabla}^{II}: \Hoch ^{II} ( \mathrm{End}_A (\cP )  )[[u]]  \to   A \ot _V \Om _{V/k}[[u]]   \end{equation}
by the same formula \eqref{formula tr nabla}.

\begin{Thm} \label{thm: tr nabla general} {\em (Theorem \ref{thm: tr nabla general appendix})} The map $\tr _{\nabla}$ satisfies 
\begin{equation*}
      (ud + d_A  - dh  ) \circ \tr _{\nabla}  =  \tr _{\nabla} \circ (uB + b_2 + b_1 + b_0) .
\end{equation*}
The same cochain map equality also holds  for $\tr _{\nabla}^{II}$.
\end{Thm}

The proof  of  Theorem \ref{thm: tr nabla general} is rather technical and will be given in Appendix \ref{sec: proof of main} closely following the proof of \cite[Theorem 5.19]{BW}.

\subsection{Proof of Theorem~\ref{thm: tr nabla}} \label{sub: compatibility} 
As in \S~\ref{sub: Cech model} we fix a finite open affine covering $\fU = \{ U _i \}_{i \in I}$ of $X$ and fix a total ordering of $I$.

From  Example \ref{example: V-con} for each affine open subset $U$ of $X$ we consider the case:
\[ V = \Gamma (U, \cO _X), A = \Gamma (U, \TC (\cO _X)), d_A = \dC, h = w\]
with the $V$-connection $d$ on $A$ from the de Rham differential of $V$.

 For each $P \in q\MFtc (X, w)$, we choose, once and for all, a $k$-linear sheaf homomorphism
\[ \nabla _P  : \TC (P) \to \TC (P) \ot  _{\cO_X}\Omega ^1_{X/k} \]
such that $\nabla _P (U)$ is an $\cO_X (U)$-connection with respect to $d$.
Such a $\nabla _P$ we call an {\em $\cO_X$-connection} on $\TC (P)$ and will construct in \S~\ref{sub: Th nabla} as $\nabla _E$ in \eqref{eqn: nabla E}  for $E=P$.
 The supertrace map  \eqref{eqn: def trace} becomes
 \begin{equation}\label{eqn: supertrace for P} 
 \tr : \Gamma (U,   \Omega ^{\bullet}_{X/k}  \ot \cE nd _{\cO _X} (P) \ot  \vC (\cO_X) )  \xrightarrow{} \Gamma (U,  \Omega ^{\bullet}_{X/k} \ot  \vC (\cO_X))  . \end{equation}

Since  any finite sum of objects of $q\MFtc (U, w|_U))$ is allowed,
 the definition of $\tr _{\nabla}^{II}$ in \eqref{eqn: def tr nabla II}  can be modified to give the map: 
\[ \overline{\tr} ^{II}_{\nabla, U}: \oHoch^{II} (q\MFtc (U, w|_U)) [[u]] \to    \Om _{X/k} (U)  \ot_{\cO_X (U)}  \TC (\cO_X) (U) [[u]]  . \]
By Theorem~\ref{thm: tr nabla general}, it is a cochain map.  Hence its sheafified version 
\[ \overline{\underline{\tr}}^{II} _{\nabla} : \overline{\underline{\Hoch}}^{II} (q\uMFtc (X, w)) [[u]]  \to \Om _{X/k}  \ot_{\cO_X}  \TC (\cO_X) [[u]]  \] 
 is also  a cochain map between the sheaves of complexes over $k[[u]]$.

Taking Lemma \ref{lem: Mix to Neg} into account,  we see that $\overline{\underline{\tr}}^{II}_{\nabla}$
fits into the following commuting diagram of  cochain maps between the sheaves of complexes over $k[[u]]$:
 \begin{align}\label{diag: tr Hoch} \\ \nonumber 
 \tiny{
\xymatrix{   \underline{\oHoch}^{II} (q\uMFtc (X, w)) [[u]]  \ar@/^1pc/[rdd]_{\ \ \ \ \ \ \overline{\underline{\tr}}^{II}_{\nabla}  }  
       &  \ar[l]_{ \sim}^{\ \ \ \ \ \S \ref{sub: cdg invariances} \eqref{item: qYoneda} \qquad } \ar[dd]_{}^{\overline{\underline{\tr}}^{II} _d}  
\underline{\oHoch}^{II} (\sAtw) [[u]]  
             & \ar[l]_{\qquad \sim}^{\qquad \S \ref{sub: cdg invariances} \eqref{item: O to CO} }  \underline{\oHoch}^{II} (\sA_w) [[u]] 
             \ar[d]_{\underline{I}_{\scriptscriptstyle{\text{HKR}}} }^{\S \ref{sub: cdg invariances} \eqref{item: HKR} } \\
\underline{\oHoch}^{II} (\uMFtc (X, w)) [[u]] \ar[u]^{inc}_{\S \ref{sub: cdg invariances} \eqref{item: Yoneda} }  &  & (\Om _{X/k}  [[u]] ,  -  dw + ud  )  \ar[ld]_{\sim} \\
& (\Om _{X/k} \ot_{\cO_X} \TC(\cO_X) [[u]] , \dC - dw + ud )           &   } }
\end{align}

\begin{Prop}\label{prop: sheaf tr nabla}
The map $\overline{\underline{\tr}}^{II}_{\nabla}$ and the map
 \[     \RR\Gamma ( \overline{\underline{\tr}}^{II}_{\nabla} \circ inc ) : 
  \RR \Gamma  (\underline{\overline{\Hoch}} ^{II} (\MFtc (X, w) ) ) [[u]]  \to           \RR\Gamma (\vC (  \Omega ^{\bullet}_{X/k} ) ) [[u]]  \]
 are quasi-isomorphisms.
\end{Prop}

\begin{proof}
It follows from Proposition \ref{prop: Efimov results} \eqref{item: Efimov results 2} and the fact that in  diagram \eqref{diag: tr Hoch} all the arrows but $\overline{\underline{\tr}}^{II}_{\nabla}$ are shown to be quasi-isomorphisms in \S~\ref{sub: cdg invariances}.
\end{proof}

We  define a $k[[u]]$-linear cochain map
\begin{equation}\label{int: def tr nabla body}   \tr _{\nabla } : \Hoch (\MFtc (X, w) )[[u]] \to \vC (\fU, \Om _{X/k} )  [[u]]  \end{equation}
by composing
\begin{align}\label{eqn: tr nabla II}
 \Hoch (\MFtc (X, w) )[[u]]  & \xrightarrow{}  \overline{\Hoch} ^{II} (\MFtc (X, w) )[[u]]    \\  \nonumber
&  \xrightarrow{natural} \Gamma (X,  \overline{\uHoch} ^{II} (\underline{D}_{\vC} (X, w) )) [[u]]   \\  \nonumber
&  \xrightarrow{\Gamma( \underline{\overline{\tr}}^{II}_{\nabla} \circ inc ) } \Gamma (  X, \vC (\Om _{X/k} ))  [[u]].
\end{align}

\begin{Thm} \label{thm: tr nabla body} {\em (Theorem \ref{thm: tr nabla})}
The $k[[u]]$-linear map
\begin{multline}\label{eqn: thm tr nabla body}
  \tr _{\nabla} :  \ \   (\Hoch (\MFtc (X, w) )[[u]], b + uB  )   \\ 
\to   (\vC (\fU, \Omega ^{\bullet}_{X/k}  ) [[u]] ,   \dC - dw + ud )
 \end{multline}
is a quasi-isomorphism  compatible with the HKR-type isomorphism.
\end{Thm}
\begin{proof} The proof follows from the following commuting diagram of cochain maps 
 \begin{equation}\label{diag: derived tr Hoch} 
 \xymatrix{ \Hoch (\MFtc (X, w) )[[u]]  \ar[r]^{\sim}_{\eqref{eqn: quot I}}  \ar@/_3pc/[ddr]_{ \tr _{\nabla} } 
 & \overline{\Hoch} (\MFtc (X, w) )[[u]]  \ar[r]^{\sim \quad }_{ \text{Prop. \eqref{prop: Efimov results}}    \quad }  \ar[d]
             & \RR\Gamma ( \underline{\overline{\Hoch}}  (\MFtc (X, w) ))  [[u]]       \ar[d]_{\sim}^{     \eqref{eqn: BW 3.9} \& \eqref{eqn: I to II}  } \\ 
          & \Gamma      ( \underline{\overline{\Hoch}} ^{II} (\MFtc (X, w) ))  [[u]]    \ar[d]_{\Gamma( \overline{\underline{\tr}} _{\nabla}^{II} \circ inc) } \ar[r] 
          &  \RR \Gamma  (\underline{\overline{\Hoch}} ^{II} (\MFtc (X, w) ) ) [[u]]     \ar[d]_{\sim}^{\text{Prop.} \ref{prop: sheaf tr nabla}}  \\    
            & \Gamma (\vC (  \Omega ^{\bullet}_{X/k} )) [[u]] \ar[r]^{\sim} & \RR\Gamma (\vC (  \Omega ^{\bullet}_{X/k} ) ) [[u]] }
 \end{equation}
 where the arrows with $\sim$  are meant to be quasi-isomorphisms. \end{proof}

 For  $ \ka _j \in \vC  (\fU, \cH om_{\cO _X} (P_{j+1} ,  P_{j} ) ) $ with $P_{n+1} =P_0$, we note that  
 $\tr _{\nabla} (\ka _0 [ \ka _1 | ... | \ka _n] ) $ is as in  formula \eqref{eqn: tr formula}
\begin{equation*} \label{eqn: tr formula body}
\sum_{(j_0, ..., j_n)} (-1)^{J}   \tr\big( \frac{ \ka _0 R_1^ { j_0 } [\nabla , \ka _1] R_2^{j_{1}} [\nabla, \ka _2]  ...  R_n^{j_{n-1}}  [\nabla  , \ka _n] R_0^{j_n}  }{(n+J )!} \big) ,
 \end{equation*}
 where $j_i \in \ZZ_{\ge 0}$ and $J= \sum _{i=0}^n j_i$.
From Remark \ref{rmk: nabla2} we see that 
\begin{align}\label{eqn: concrete form}  [\nabla , \ka _j] & =   \prod _ {i_0 < ... < i_{p_j}}    \big( \nabla _{P_j , i_0}|_{U_{i_0, ..., i_{p_j}}} \circ   (\ka_j)_{i_0, ..., i_{p_j}} -
 \\ \nonumber 
  & \ \ \ \ \ \ \ \ \  \ \ \ \ \ \ \ \ \   \ \ \ \ \ \ \ \ \   (-1)^{|\ka _j|  } (\ka_j)_{i_0, ..., i_{p_j}} \circ   \nabla _{P_{j+1} , i_0}|_{U_{i_0, ..., i_{p_j}}} \big)    ;   \\  \nonumber
R_j & =  \prod _{i} \left( u \nabla _{P_j, i} ^2  + 
  [ \nabla _{P _j , i} ,  \delta _{P _j} |_{U_{i}} ]   \right) + \prod _{i_0 < i_1} ( \nabla _{P _j , i_0} -   \nabla _{P _j , i_1} )  .
\end{align}

\subsection{Construction of $\cO_X$-connections on $\TC ( E ) $}\label{sub: Th nabla}

Let $E$ be a vector bundle on $X$.
Choose a connection $\nabla _i $ of $E|_{U_i}$ for $i\in I$:
\[\nabla _i : E|_{U_i} \to  \Omega ^1_{U_i/k}   \ot  _{\cO_{U_i}} E|_{U_i} , \]
which exists, since $U_i$ is affine.
We consider 
\begin{equation}\label{eqn: nabla E} \nabla _E := \prod _{i \in I} \nabla _i \in \vC ^0 (\fU, \cH om _{k} (\vC (E), \Omega ^1_{X/k} \ot \vC (E)  ))  \end{equation}
a  \vC ech element of a {\em $k$-sheaf}  $\cH om _{k} (\vC (E), \Omega ^1_{X/k} \ot \vC (E) )$ on $X$.
Note that $\nabla _E$  satisfies \[ \nabla _E ( s a ) = \nabla _E (s) a + (-1)^{|s|}s\ot d a \]
for $s \in \TC (E)$, $a \in \TC (\cO _X)$. 
Thus for each affine $U$ open subset of $X$,  $\nabla _E (U)$  is  an $\cO_X (U)$-connection of  the $\cO_X (U)$-module $\vC (E) (U)$ with respect to $d$.
We have constructed an $\cO_X$-connection on $\TC (E)$.

\begin{Rmk}\label{rmk: nabla2}
Note that 
\begin{align*}
\nabla _E ^2 & \in   \Hom_{\CO } (\TC (E),\Omega ^2_{X/k} \ot \vC (E) )   =  \Hom_{\cO _X } (E, \Omega ^2_{X/k} \ot \vC (E) ) ; \\
    [\nabla _E , \dC]   & \in   \Hom_{\CO } (\TC (E), \Omega ^1_{X/k} \ot \vC (E))  =  \Hom_{\cO _X } (E, \Omega ^1_{X/k} \ot \vC (E) )     .       \end{align*}
In fact they are elements in   $ \Hom_{\cO _X } (E, \Omega ^{2}_{X/k}  \ot _{\cO _X} \TC ^0 (E)  )$ and $  \Hom_{\cO _X } (E,  \Omega ^{1}_{X/k} \ot _{\cO _X} \TC ^1 (E)   ) $, respectively.
Now it is easy to see that they are expressed as 
\begin{align} \label{eqn: usual curvature} \prod _i  \nabla _i ^2  & \in  \vC ^0 (\fU,  \Omega _{X/k} ^2  \ot _{\cO _X} \cE nd_{\cO _X} (E) ) ; \\ \label{eqn: usual Atiyah}
\prod _{i < j} (\nabla _i - \nabla _j )   & \in    \vC ^1 (\fU, \Omega _{X/k} ^1     \ot _{\cO _X}    \cE nd_{\cO _X} (E) )  . \end{align}
Let $\ka \in \vC ^p ( \fU, \cE nd_{\cO _X} (E)) \subseteq \End_{\CO }  (\vC (E))  $.
Then $[\nabla _E, \ka ]$ is an element in  $ \Hom_{\CO} (\TC (E) ,\Omega ^1 _{X/k} \ot _{\cO _X}  \TC (E)   )$, which is expressed as 
\[ \prod _{i_0 < ... < i_p} [\nabla _{i_0}|_{U_{i_0, ..., i_p}}  , \ka _{i_0 ,... , i_p}  ] \in \vC ^p (\fU,  \Omega _{X/k} ^1 \ot _{\cO _X} \cE nd_{\cO _X} (E)   ) . \]
\end{Rmk}

\begin{Rmk}\label{rmk: some exp} Note that  \eqref{eqn: usual curvature} is  the curvature $\nabla _{E}^2$ of the connection $\nabla _E$ on $\vC (E)$ 
and  \eqref{eqn: usual Atiyah} is   the \vC ech representative of
the Atiyah class of the vector bundle $E$ (see \cite{Huyb} for example).
\end{Rmk}

\section{Applications}\label{sec: app}
In this section we will apply Theorem~\ref{thm: tr nabla} to prove 
 a Chern character formula~\eqref{eqn: ch cech form}.
We also generalize the formula to the localized case as well as the equivariant case of a finite group action.

\subsection{A Chern character formula}
Denote by $\ch_{HN} (P)$ and $\ch _{HH} (P)$ 
 the classes in \[  \HH^{-*} (X, (\Om _{X/k} [[u]], ud - dw )) \text{ and }   \  \HH^{-*} (X, (\Om _{X/k} ,  - dw ))  \] 
 corresponding to  $\Ch _{HN} (P)$ and  $\Ch _{HH} (P)$, respectively via  the isomorphism induced from \eqref{eqn: thm tr nabla body}.

\subsubsection{Proof of Theorem~\ref{thm: ch cech form}} 
Recall that under the natural cochain map there is an isomorphism
\[   \overline{HN}_*  (\MFtc (X, w)) \xrightarrow{\cong} \overline{HN}_*^{II} (\MFtc (X, w)) . \]
Now considering the  class $[1_P]$ of $\overline{HN}_*^{II} (\MFtc (X, w))$ and denoting by $ \overline{\tr} _{\nabla , X }^{II}$  the composition of the last two maps in \eqref{eqn: tr nabla II},
 we get \begin{equation}\label{eqn: ch tr 1}  
\ch _{HN} (P)=  \overline{\tr} _{\nabla , X }^{II} ( 1_P )  \end{equation}
in the cohomology $\check{H}^0 (\fU, (\Omega ^{\bullet}_X   [[u]]  , ud   -  dw  )) .$
For  an $\cO_X$-connection $\nabla _P$ on $\CP$  let
 \begin{align}\label{eqn: total curvature}  R & :=   u\nabla _P^2   + [\nabla _P, \delta _P + \dC ] ,   
\end{align} which we call the total curvature of $\nabla _P$.
Then by  formula \eqref{formula tr nabla} we have
\begin{equation*}\label{eqn: ch form}  \overline{\tr} _{\nabla , X}^{II} ( 1_P )   = \tr  \exp ( - R) .  \end{equation*}
Therefore we have established the following.
\begin{Thm}\label{thm: ch cech form body} {\em (Theorem \ref{thm: ch cech form})}
The Chern character  $\chHN (P)$ is representable by 
\begin{equation}\label{eqn: ch cech form body} \chHN (P) = \tr  \exp  (-R)   . \end{equation}
\end{Thm} 

Remark~\ref{rmk: nabla2} shows  expression \eqref{eqn: formula R}.

\subsubsection{When $k = \CC$}\label{sub: C case}
Assume furthermore that the critical locus of $w$ is proper over $\CC$.
Let $X^{an}$ denote the complex manifold associated to $X$.
Due to GAGA, instead of $\TC (P)$, we may use the Dolbeault resolution  \[ (\sA ^{0, \bullet} _{X^{an}}, \bar{\partial})    \ot _{\cO_{X^{an}}}      P^{an} \] 
to construct the Dolbeault dg enhancement $\Perfd (X, w)$ of the derived category $D (X, w)$. Here $P^{an}$ denotes the holomorphic version of the matrix factorization $P$.
The objects of $\Perfd (X, w)$ are the objects of $D_{dg} (X, w)$. However $\Hom (P, Q)$ 
in $\Perfd (X, w)$ is, by definition, a complex  \[\Gamma (X,    \sA ^{0, \bullet} _{X^{an}}     \ot_{\cO _{X ^{an}}}     \cH om _{\cO _{X^{an}}}(P^{an}, Q^{an})  ). \] 
It is straightforward to check that the Dolbeault version of Theorem~\ref{thm: tr nabla body} holds once we choose a differentiable $(1, 0)$-connection $\nabla$ for each $P \in  \Perfd (X, w)$.
Hence we also have a Chern character formula 
\begin{equation}\label{eqn: ch form an} \chHN (P) = \tr  \exp ( -  u\nabla ^2   - [\nabla, \delta _P + \bar{\partial} ])  \end{equation}
in the cohomology 
\[H^0 (\Gamma (X^{an},   \sA ^{\bullet, \bullet} _{X^{an}}   ) [[u]] , u\partial  + \bar{\partial} -  \partial w\wedge  ) . \]

\subsection{Cohomology with supports}\label{sec: local}
A matrix factorization $P$ for $(X, w)$ is called {\em acyclic} if
 $1_{P}$ is null-homotopic in $\End _{\MFdg (X, w)} (P)$. 
Let $Z$ be a closed subset of $X$. 
We say that  a matrix factorization $P$ for $(X, w)$ is {\em supported} on $Z$ if
$P |_{X-Z}$ is an acyclic matrix factorization for $(X-Z, w|_{X-Z})$.
We may consider a dg category $D_{dg}(X, w)_Z$  of matrix factorizations for $(X, w)$ supported on $Z$. It is a full subcategory of $D_{dg} (X, w)$.

Let $\fU_1 = \{ U_i \} _{i \in I_1} $ be a finite collection of open affine subsets $U_i$ of $X$ such that 
\[\bigcup_{i\in I_1}  U_i \supseteq Z  \]  and 
let $\fU_2 = \{ U_i \} _{i\in I_2} $ be a finite open affine  covering of $X-Z$.
Let $I :=I_1 \coprod I_2$ and $\fU := \{ U_i \} _{i \in I}$, a finite open affine covering of $X$.
As before  $\CP$ and $\vC (P|_{X-Z})$ denote  the ordered \v{C}ech complexes  of $P$ and $P|_{X-Z}$ with respect to $\fU$ and $\fU_2$, respectively.
Let $\CP _Z$ denote the kernel of the natural projection cochain map $\CP \to j_* \vC (P|_{X-Z})$ where $j: X-Z \to X$ is the open inclusion.

With respect to $\fU$, we have the \vC ech model  $\MFtc (X, w)$ for $D_{dg}(X, w)$; see \S~\ref{def: MFtc}.
We define $\MFtc (X, w)_Z$ as the full subcategory of  $\MFtc (X, w)$ consisting all objects supported on $Z$.
We will define another version $\MFtc (X, w)^{rel}_Z$ of $D_{dg}(X, w)_Z$, which is a relative version.
The objects remain the same. However
the Hom space $\Hom (P, Q)$ for $P, Q \in \MFtc (X, w)^{rel}_Z$ is defined to be 
\begin{equation}\label{eqn: rel Hom} \Hom _{\CO _Z }^{\bullet} (\CP _Z, \CQ _Z) \cong \Gamma (X,   \cH om _{\cO _X}^{\bullet} (P,  Q) \ot _{\cO_X}  \CO _Z     ) . \end{equation}
 This is a relative version of \eqref{eqn: Hom}. Note that the natural inclusion \[\Hom _{\CO _Z }^{\bullet} (\CP _Z, \CQ _Z) \to \Hom _{\CO  }^{\bullet} (\CP , \CQ ) \]
 is a quasi-isomorphism. The relative version  $\MFtc (X, w)^{rel}_Z$ is not necessarily unital, which is required in the definition of dg categories at \S \ref{sub: dg k cat}.
It is however cohomologically unital, i.e., it has units in its cohomological category; for the definition see \cite[\S 1,1]{Sei}. 

For a not-necessarily unital dg category $\mathscr{C}$, with no changes the definition of Hochschild complex $(C (\mathscr{C} ), b)$ works.
There is another version, the so-called non-unital Hochschild complex \[ (C^e (\mathscr{C}) , b) \]  which includes $(C (\mathscr{C}) , b)$ as a subcomplex and
the inclusion is a quasi-isomorphism.
The complex  $(C^e (\mathscr{C}) , b)$  however has a suitable Connes differential $B^e$ making a mixed complex $(C^e (\mathscr{C}) , b^e, B^e)$; 
see \cite[\S 3.5]{She} where $C^e(\mathscr{C})$ is denoted by $CC^{nu}_{\bullet} (\mathscr{C})$. 
On the other hand for a unital dg category $\cA$, there is a quasi-isomorphism 
\[ p(u) : (C^e(\cA) [[u]] , b^e + uB^e) \xrightarrow{\sim} (C(\cA) [[u ]], b + uB) \]  defined in  \cite{Shk: NC}.

Therefore we may consider a natural commutative diagram in the category of cochain complexes
 \begin{equation}\label{eqn: Delta local} {\tiny  \xymatrix{  
  \Hoch ^e (\MFtc (X, w)^{rel}_Z )[[u]]  \ar@{^{(}->}[d]_{inc}  \ar@/_4pc/@{-->}[ddd]_{\tr_{\nabla}^Z}  & &  \\
 \Hoch ^e (\MFtc (X, w)_Z )[[u]] \ar[r] \ar[d]_{p(u)} 
& \Hoch ^e (\MFtc (X, w) )[[u]] \ar[r] \ar[d]_{p(u)} & \Hoch ^e (\MFtc (U, w|_{U} ))[[u]]  \ar[d]_{p(u)} \\
 \Hoch (\MFtc (X, w)_Z )[[u]] \ar[r] 
& \Hoch (\MFtc (X, w) )[[u]] \ar[r] \ar[d]_{\tr _{\nabla, X}} & \Hoch (\MFtc (U, w|_{U} ))[[u]]  \ar[d]_{\tr _{\nabla, U}} \\
                  \Gamma (X,  \TC (\Omega ^{\bullet}_{X/k})_Z [[u]] ) \ar@{^{(}->}[r] &   \Gamma (X,  \TC (\Omega ^{\bullet}_{X/k}) [[u]] ) \ar@{->>}[r]  &   \Gamma (U, \TC (\Omega ^{\bullet}_{U/k})[[u]] ) .  }  } \end{equation} 
                                    Here $\tr _{\nabla, X}$ (resp. $\tr _{\nabla, U}$) denotes $\tr _{\nabla}$
                  for $X$ with respect to $\fU$ (resp. for $U$ with respect to $\fU _2$) and
                  the dotted arrow $\tr _{\nabla}^Z$  denotes $\tr _{\nabla, X}$ restricted to $\Hoch ^e (\MFtc (X, w)^{rel}_Z )[[u]]$, which lands on
                      the subcomplex  $  \Gamma (X,  \TC (\Omega ^{\bullet}_{X/k})_Z [[u]] )$.

\begin{Thm} \label{thm: rel tr nabla} The cochain map
\begin{multline}\label{eqn: rel tr nabla}
  \tr _{\nabla}^Z :  \ \   (\Hoch ^e (\MFtc (X, w)^{rel}_Z )[[u]], b^e + uB^e  )   \\ 
\to   (\Gamma (X,  \Omega ^{\bullet}_{X/k} \ot_{\cO_X} \TC (\cO_X)_Z   ) [[u]] ,   \dC - dw + ud )
 \end{multline}
is a quasi-isomorphism  compatible with the HKR-type isomorphism $\tr _{\nabla, X}$.  Furthermore,
in the relative \v{C}ech cohomology \[H^0 (\Gamma (X, \Omega ^{\bullet}_{X/k} \ot_{\cO_X} \TC (\cO_X)_Z   ) [[u]]  ,   \dC - dw + ud  ), \] 
the localized Chern character $\ch^{Z}_{HN} (P)$ for $P \in \MFtc (X, w)_Z $ with local connections \eqref{eqn: nabla i} is representable by
\begin{equation}\label{eqn: rel ch cech form}  
\tr  \left(\prod _{i\in I_1} 1_{P|_{U_i}} \cdot \exp \left( -  u(\nabla_i ^2)_{i\in I}   -( [\nabla_i, \delta _P])_{i \in I}  - (\nabla _i - \nabla _j )_{i < j,\ i , j \in I} \right)\right) .
\end{equation}
\end{Thm}

\begin{proof} First we invoke two facts that $inc$ in \eqref{eqn: Delta local} is a quasi-isomorphism; see \cite[Corollary 4.13]{She} and 
the third line of \eqref{eqn: Delta local} is a distinguished triangle; see \cite[\S 5.7]{Keller_Cyclic}.
Note that the last line of \eqref{eqn: Delta local} is obviously a short exact sequence.
By these facts together with the quasi-isomorphisms $p(u)$
it follows that  $ \tr _{\nabla}^Z$ is also a quasi-isomorphism. 
Since the unit $\prod _{i\in I} 1_{P|_{U_i}}$ for $P$ in $\MFtc (X, w)_Z$ is homologous to $1_P^Z :=\prod _{i\in I_1} 1_{P|_{U_i}} $ and, by the definition of $p(u)$,
\[ p(u) \circ inc  (1_P^Z) = 1_P^Z ,\]
we conclude \eqref{eqn: rel ch cech form}.
\end{proof}

\subsection{Global quotients by a finite group}\label{sub: orb}
Let a finite group $G$ act on a smooth quasi-projective variety $X$ from the  left.  Suppose that $w$ is $G$-fixed.
We consider the dg category $ \MFGdg (X, w) $  of $G$-equivariant matrix factorizations for $(X, w)$.
The Hom dg-space  from $P$ to $Q$  is the $G$-fixed part of $\Hom _{\MFtc (X , w)} (P, Q)$.

Denote by $\frt$ the action map and by $\frs$ the projection map
\[ \frt: G \ti X \to X , \ (g, x) \mapsto g  x ; \ \ \ \frs: G\ti X  \to X , \ (g,  x) \mapsto x .\]
 For $E \in \MFtc (X, w)$ denote $(g^{-1})^*E $ by $gE$ and let
 \[\widetilde{E} := \frs_* \frt^* E = \bigoplus _{g \in G}  g E  = E \ot \bigoplus _g g\cO _X .\]  
  We will consider  $\widetilde{E}$ as an object 
 in $\MFGdg (X, w)$ endowed with the natural $G$-equivariant structure.

 Let $\tMF (X, w)$ be the full dg-subcategory of $\MFGdg (X, w)$
consisting of all the objects of form $\widetilde{E}$ for some $E \in \MFtc (X, w)$. 
In fact this gives rise to a functor $ \MFtc (X, w) \to \tMF (X, w)$.
We will see that every object $P$ of $\MFGdg (X, w)$ is a direct summand of $\widetilde{P}$ as follows.

For  $P\in \MFGdg (X, w)$, its $G$-equivariant structure is a compatible isomorphism $\varphi : \frt ^* P \to \frs ^* P$.
Let $\varphi _g$ be the $g$-component of $\varphi$, 
\[  \varphi _g   :  g P   \xrightarrow{\cong}    P.   \] 
Then we get a  map $\iota : P \to  \frs _* \frt^* P $ given by the composition 
\[ \iota: P \xrightarrow{\text{natural}}  \frs _* \frs ^* P  \xrightarrow{\frs _*\varphi ^{-1} } \frs _* \frt^* P .  \] 
Let $\varpi $ be a  map $ \frs _* \frt^* P \to P$  given by the composition 
\[ \varpi :  \frs _* \frt ^* P \xrightarrow{\frs_* \varphi } \frs _* \frs ^* P \xrightarrow{\frac{\text{sum}}{|G|}} P ,  \]
making $\varpi \circ \iota = 1_P$. 
We impose a $G$-equivariant structure  on  $ \frs _* \frt^* P  $ by letting 
\[      \frs _* \frt^* P  = \widetilde{P ^{\sharp}}  = P   \ot \bigoplus _g g\cO _X \]  where
 $P^{\sharp}$ denotes the object in $\MFtc (X, w)$ corresponding to $P$. 
Then it is straightforward to check that $\iota $, $\varpi$ are closed maps in $\MFGdg (X, w)$.
By Morita invariance \eqref{eqn: mor} we have the following lemma.
\begin{lemma}
The inclusion
\[ \MC (\MFGdg (X, w))  \xleftarrow{inc} \MC (\tMF (X, w)) \]  is a quasi-isomorphism.
\end{lemma}

For $E, F \in \MFtc (X, w)$ every map $\phi : \widetilde{F} \to \widetilde{E}$ in $\tMF (X, w)$ is expressed as a square matrix
$(\phi _{g, h})_{g, h \in G}$ whose entries $\phi _{g, h} : gF \to hE$ is a morphism in $\MFtc (X, w)$. Since $\phi$ is $G$-equivariant,
$h ( \phi _{h^{-1} g , \mathrm{id}} ) = \phi _{g, h}$ if we let $h ( \phi _{h^{-1} g , \mathrm{id}} ) := (h^{-1})^* ( \phi _{h^{-1} g , \mathrm{id}} )$.
Hence $\phi$ is uniquely determined by
the components \[\phi _g  := \phi _{g, \mathrm{id}}:  gF \to E , \  g \in G . \] 
We formally write $\phi = \sum _{g\in G}  \phi _g \ot g$.

\begin{example}
For $\ka \in \End ^G_{\cO_X} (P)$, let $\widetilde{\ka} := \iota \circ \ka \circ \varpi$, which is an endomorphism 
of $\widetilde{P^{\sharp}} = \frs _* \frt^* P$ in $\MFGdg (X, w)$. 
Since $\widetilde{\ka}$ is a $G$-equivariant map, it is determined by each component
$gP \to P$, which is 
\begin{equation}\label{eqn: g} \frac{1 }{ |G|}  \ka \circ \varphi _g . \end{equation} 
\end{example}

 For $E_i \in \MFtc (X, w)$ 
the  composition  $(a_0 \ot g_0) \circ (a_1\ot g_1) $ of $a_0\ot g_0 : \widetilde{E_1} \to \widetilde{E_0}$ and $a_1 \ot g_1 : \widetilde{E_2} \to \widetilde{E_1}$ 
is written as $ a_0 \circ g_0(a_1) \ot g_0g_1$, which is illustrated in the diagram
\[   E_0 \xleftarrow{a_0} g_0 E_1  \xleftarrow{ g_0 (a_1) } g_0 g_1 E_2    \]  where 
 $g_0(a_1) := (g_0^{-1})^* a_1$.

Let $X^g$ denote the $g$-fixed loci of $X$ and let $w_g := w |_{X^g}$. Consider a $G$-action on 
the disjoint union $\coprod _{g \in G} X^g$ by
\[ X^g \to X^{hgh^{-1}}, \ x \mapsto h x , \text{ for } h \in G .  \]
It induces an action on the Hochschild complex of $\coprod _{g \in G} X^g$ and its $G$-coinvariants
are denoted by $ (\oplus _{g \in G} \Hoch (\MFtc (X ^g , w_g) ))_G$.
In the category of mixed complexes, following Baranovsky \cite{Bar}, we  define a morphism 
\[  \Hoch(\tMF (X, w)) \xrightarrow{\Psi }    
                            (\oplus _{g \in G} \Hoch (\MFtc (X ^g , w_g) ))_G   \]
                       by sending
             \begin{equation*}  a_0 \ot g_0  [ a_1 \ot g_1  | \cdots  | a_n \ot g_n ]     \mapsto
                a_0{_ {|_{X^g} }}   [  g_0 (a_1)_ {|_{X^g}}  |  g_0g_1 (a_2 )_{|_{X^g}}   | 
              \cdots | g_0 ... g_n (a_{n})_{|_{X^g}}  ]  \end{equation*}
                        with $g := g_0 ... g_n$. 
                        
                        \begin{Prop}
                        The morphism $\Psi$ between the mixed complexes is a quasi-isomorphism.
                        \end{Prop}
          \begin{proof}              
                        When $w=0$, the map $\Psi$ was defined and shown to be a quasi-isomorphism in \cite{Bar}.
                        For general $w$ we may assume, by the Mayer-Vietoris sequence argument,  that $X$ is affine.
                        In this case, for the Hochschild homology of the second kind, it is  a quasi-isomorphism by a spectral sequence argument as in \cite[Theorem~6.3]{CT}. 
                        Now again by  \eqref{eqn: I to II}, we conclude that it is  a quasi-isomorphism in the Hochschild complex of the  first kind. Taking Lemma
                        \ref{lem: Mix to Neg} into account, we complete the proof.
            \end{proof}         

Fix a finite open affine covering $\fU = \{ U_i \} _{i\in I}$ of $X$ such that each $U_i$ is $G$-invariant. This  induces  a finite open affine covering $\{ X^g \cap U_i  \}_{i\in I}$ 
of $X^g$. 
For each $g$ and each $Q \in \MFtc (X^g, w)$, we choose, once and for all, a connection $\nabla _{Q, i } $ of $Q|_{X^g\cap U_i }$ such 
that the chosen connection $\nabla _{h^*Q, i}$ of $(h^*Q)|_{X^{h^{-1} gh} \cap U_i }$ is isomorphic to  $h^* (\nabla _{Q, i})$. This is possible by  Zorn's lemma. 
Now we have a quasi-isomorphism 
\begin{equation}\label{eqn: tr nabla equiv} 
    \oplus _{g \in G} \Hoch (\MFtc (X ^g , w_g) ) [[u]]  \xrightarrow{\tr _{\nabla}} \oplus _{g \in G}  \Gamma (X^g,  \TC (\cO_{X^g}) \ot \Omega ^{\bullet}_{X^g/k} )[[u]]   
    \end{equation} 
defined by applying $\tr _{\nabla}$ component-wisely for each $g$. The cochain map $\tr _{\nabla}$ in \eqref{eqn: tr nabla equiv}  is $G$-equivariant.
Hence it induces a cochain map in $G$-coinvariants, also denoted  by $\tr_{\nabla}$ by abuse of notation.

\begin{Thm}\label{thm: eq ch}
There is a natural isomorphism
\begin{equation}\label{eqn: equiv HN}  HN_*  (\MFGdg (X, w)) \cong  (\oplus _{g \in G}   H ^{-*} (\Gamma ( X^g,  \vC(\Omega ^{\bullet}_{X^g/k})) [[u]],  \dC - dw_g + ud ))_G . \end{equation}
Under the isomorphism, the $G$-equivariant Chern character becomes
\begin{equation}\label{eqn: equiv ch form} \ch^G_{HN} (P) = \tr _{\nabla} \Psi (\eta _{\pi})  , \end{equation}
where $\pi :=\iota \circ \varpi$ and $\eta _{\pi}$ is from \eqref{eqn: eta pi}.
In particular, 
\begin{equation}\label{eqn: equiv HH ch form} 
\ch ^G_{HH} (P) = \frac{1}{|G|} \bigoplus _{g \in G}  \tr  \left( \varphi _{g |_{X^g}}  \exp ( -  [ \prod _{i\in I} \nabla _{P|_{X^g}, i },  \dC + \delta _{P|_{X^g}}  ]  ) \right) \end{equation}
for $G$-equivariant connections $\nabla _{P|_{U_ i}}$ of $P|_{U_i}$. 
\end{Thm}

\begin{proof}
The isomorphism \eqref{eqn: equiv HN} follows by isomorphisms
\[ HN_*  (\MFGdg (X, w)) \xleftarrow{} HN_*( \tMF (X, w)) \xrightarrow{(\tr_{\nabla} \circ \Psi)_*}    \text {RHS of \eqref{eqn: equiv HN}} .\] 
The specialization of \eqref{eqn: equiv ch form} at $u=0$ becomes  \eqref{eqn: equiv HH ch form}  by \eqref{eqn: g}.  
Thus it is enough to show \eqref{eqn: equiv ch form}. However this 
is clear from Proposition~\ref{prop: eta} which says that $\eta _{\pi}$ is homologous to $1_P$ in $(\overline{C}  (\MFGdg (X, w))[[u]], b+ uB )$.
\end{proof}

\begin{Rmk}
Since $\ka$ is homologous to $\widetilde{\ka} = \sum _{g \in G}  \frac{1}{|G|} ( \ka \circ \varphi _{g})  \ot g $, we  also get a boundary bulk map formula for Hochschild homology:
\begin{equation}\label{eqn: equiv HH bb} \ka \mapsto \frac{1}{|G|} \bigoplus _{g \in G}  \tr  \left( (\ka \circ \varphi _g )_{|_{X^g}}  
 \exp (  -  [ (\nabla _{P|_{X^g}, i })_{i\in I}  , \dC + \delta _{P|_{X^g}}]) \right) . \end{equation}
\end{Rmk}

\begin{Rmk}
Let $X$ be an open subscheme of $\mathbb{A}^n_{k}$ containing the origin and let $w$ have only one critical point at the origin. 
Under the natural isomorphism $(\bigoplus _{g\in G} \Omega ^{\bullet}_{X^g/k})^G \to (\bigoplus _{g\in G} \Omega ^{\bullet}_{X^g/k} )_G$
and  the completion,
\eqref{eqn: equiv HH bb} coincides with  the corresponding formula~(3.16) in \cite[Theorem 3.3.3]{PV: HRR} up to a  normalization  factor $|G|$ and 
the sign convention in HKR-type isomorphisms.
\end{Rmk}

\begin{Rmk}
When $k=\CC$, using a $G$-equivariant differentiable $(1,0)$-connection on $P$ we get a $G$-equivariant Chern character formula of $\ch ^G_{HH} (P)$ 
 in the Dolbeault cohomology. It is the $G$-equivariant version of \eqref{eqn: ch form an}.
\end{Rmk}

\subsubsection{The equivariant mixed complex} Let $\overline{\Psi} ^{II}$ be the version of $\Psi$ for the mixed normalized Hochschild complex of the second kind.
We have a diagram of quasi-isomorphisms of mixed complexes 
 \begin{equation}\label{diag: equiv Hoch} 
  \xymatrix{\overline{C}^{II} (\MFGdg (X, w)) & \ar[l]_{inc}  \overline{C}^{II} (\tMF (X, w)) \ar[ld]_{\overline{\Psi} ^{II}}   \\
    \ar[r]_{inc}     (\oplus _{g \in G} \oHoch^{II} (\MFtc (X^g, w_g)))_G  &    (\oplus _{g \in G} \oHoch^{II} (q\MFtc (X^g, w_g)))_G   \\
    \ar[ur]_{inc}  (\oplus _{g \in G} \oHoch^{II} (\Gamma (X^g, \vC( \cO _{X^g}))))_G \ar[r]_{\Ihkr} &         (\oplus _{g \in G} (\Gamma (X^g, \vC (\Omega _{X^g/k})) )_G    . }  \end{equation}   
The third inclusion map $inc$ is a quasi-isomorphism, since it is the case for the affine case.

\newcommand{\oMC}{\overline{\MC}}

\begin{Thm}\label{thm: ch orb body}  
The natural morphism $\oMC  (\MFGdg (X, w)) \to \oMC ^{II}  (\MFGdg (X, w))$ is a quasi-isomorphism and   $\oMC ^{II}  (\MFGdg (X, w)))$ is quasi-isomorphic to
\[   \left(\oplus _{g \in G}  (\Gamma ( X^g,  \vC(\Omega ^{\bullet}_{X^g /k })), \dC - dw_g , d ) \right)_G . \]
\end{Thm}

\begin{proof} Since the second statement follows from  diagram \eqref{diag: equiv Hoch}, it remains to prove the first statement.
When $X$ is affine, its proof follows from  \S \ref{sub: normalized}, \eqref{eqn: I to II}, and the first two quasi-isomorphisms in diagram \eqref{diag: equiv Hoch}.
For the general $X$, note first that there exists  a finite collection of $G$-invariant open affine schemes that cover $X$, since $X$ is quasi-projective. 
We may now apply the Mayer-Vietoris exact triangle argument as in \cite{Bar, Ef} to reduce the proof to 
the affine case.  \end{proof}

\appendix

\section{Proof of Proposition~\ref{prop: eta}}\label{sec: pf 1}

\begin{Prop}\label{prop: eta appendix} {\em (Proposition \ref{prop: eta})} The two cycles $1_P$ and $\eta _{\pi}$ are  homologous in $(\overline{C} \{ P, N \} [[u]], b+ uB)$.
\end{Prop} 
\begin{proof}
  Let $\cM$ be a $k$-category consisting of two formal objects, say $P$ and $N$, and  arrows 
  generated by morphisms $1_P := \mathrm{id} _P, g: P \to N, f: N \to P, 1_N := \mathrm{id} _N, \pi : N \to N $, together with the obvious relations by $1_P, 1_N$ and retract relations
\[ g\circ f = \pi,\  f\circ g = 1_P. \]
Let $\overline{\cM}(x, y)$ denote $\cM (x, y)$ if $x \ne y$ or
$\cM (x, y) / k \cdot 1_x $ if $x = y$.
The proof of Proposition~\ref{prop: eta appendix} amounts to finding 
$\{ \xi_i \}_{i=1}^{\infty}$ in
\[ \xi_i \in \overline{C}_{2i-1}(\cM) := \bigoplus _{x_0, ..., x_{2i-1} \in \cM} \cM (x_1, x_0) \ot \overline{\cM} (x_2, x_1)[1] \ot \cdots \ot \overline{\cM} (x_0, x_{2i-1}) [1] \] 
such that  $\xi_1:=g[f]$  and the $\xi _i$ satisfy the following relation
\begin{equation}\label{prop2.2_eq1}
    b(\xi_{i+1})=\eta_i-B(\xi_i)\text{ where }\eta_\pi:=\sum_{i=0}^\infty\eta_i u^i .
    \end{equation}
Let \[\overline{D'}:=\bigoplus_{n=0}^{\infty}\{a_0[a_1| ... | a_n]\mid
\begin{aligned}&a_0=t_{k_1}\pi+t_{k_2}1_N\text{ and }\\
  &a_{i}=t_{k_i} \pi\text{ for }t_{k_i}\in k,\text{ all }i\geq 1\end{aligned}\}\subseteq\overline{\Hoch} (\cM) . \]
  One may check that it is an acyclic subcomplex in $\overline{\Hoch} (\cM)$. Letting $\overline{\Hoch}^2 (\cM):=\overline{\Hoch} (\cM)/\overline{D'}$, define
\[\overline{D''}^2:=\bigoplus_{n=0}^{\infty}\{a_0[a_1| ... | a_n]\mid
\begin{aligned}&a_i=t_i f \text{ and }a_{i+1}=t_{i+1} \pi\\
  &\text{ for some }i+1\text{ mod }n+1\text{ and } t_j \in k\end{aligned}\}\subseteq\overline{\Hoch}^2 (\cM). \]
  By some computation, it can be also seen that it is an acyclic subcomplex in $\overline{\Hoch}^2 (\cM)$. In the quotient space
  \[\overline{\Hoch}^3 (\cM):=\overline{\Hoch}^2 (\cM)/\overline{D''}^2, \]
  relation~\eqref{prop2.2_eq1} becomes
  \begin{equation}\label{prop2.2eq2}
    b(\xi_{i+1})=-B(\xi_{i})\text { for }i\ge 1.
    \end{equation}
In $\overline{\Hoch}^3 (\cM)$, letting $\xi_1=g[f]$ and for $i\geq 2$
\[\begin{aligned}
  \xi_i:&=(i-1)!(-1)^{i-1}\bigg(g[f|\underbrace{g|...|f}_{(g|f)^{i-1}}]-\big(1_N[\pi|g|...|f]+1_N[\pi|\pi|\pi|\underbrace{g|...|f}_{(g|f)^{i-2}}]+\cdots\\
      &\cdots+1_N[\underbrace{\pi|...|\pi}_{2i-3}|g|f]\big)\bigg)\in \overline{C}_{2i-1}^3 (\cM),\end{aligned} \]
one may check that they satisfy the recursive relation~\eqref{prop2.2eq2}.
\end{proof}

\section{Proof of Theorem \ref{thm: tr nabla general}}\label{sec: proof of main}

\begin{Thm} \label{thm: tr nabla general appendix} {\em (Theorem \ref{thm: tr nabla general})}The map $\tr _{\nabla}$ satisfies 
\begin{equation*}
      (ud + d_A  - dh  ) \circ \tr _{\nabla}  =  \tr _{\nabla} \circ (uB + b_2 + b_1 + b_0) .
\end{equation*}
The same cochain map equality also holds  for $\tr _{\nabla}^{II}$.
\end{Thm}

\begin{proof} It is enough to show the equality for $\tr_{\nabla}^{II}$.
The proof of \cite[Theorem~5.19]{BW} for when $A=V$, $V$ is commutative,  and $\delta _P=0$ also works for the  general case. 
We provide some details following the proof.
In what follows, let 
\begin{equation}\label{eqn: alpha '} \ka ' := [\nabla, \ka]  \text{ for } \ka \in \End_A (\cP) , \end{equation}
 and let $\Grad : = u\nabla + \delta _{\cP}$ and $J :=\sum_{i=0}^n j_i$.
 Note that $[\delta _{\cP}, \ka ] \in \End _A (\cP )$. 
 Recalling $\delta_{\cP}=\delta _P \ot 1 + 1\ot d_A$ and the $V$-linearity of $\delta _P$, we see that
\begin{align*}
   &    (ud +d_A)\circ\mathrm{tr}_{\nabla}^{II} (\ka_0[ \ka_1| \cdots |\ka_n]) \\
    &=  \mathrm{tr}([u\nabla+\delta _{\cP} ,\sum_{(j_0, ..., j_n)}(-1)^J\frac{\ka_0R^{j_0}\ka'_1\cdots \ka'_nR^{j_n}}{(J+n)!}])\\
&=  \mathrm{tr}(\sum_{(j_0, ..., j_n)}(-1)^J\frac{[\Grad,\ka_0]R^{j_0}\ka'_1\cdots \ka'_nR^{j_n}}{(J+n)!})\\
    &+ \mathrm{tr}(\sum_{i=1}^{n}(-1)^{i-1 + \sum _{k=0}^{i-1} |\ka _k| } 
    \sum_{(j_0, ..., j_n)}(-1)^J\frac{\ka_0R^{j_0}\ka'_1\cdots R^{j_{i-1}}[\Grad,\ka '_i]R^{j_i}\cdots \ka'_nR^{j_n}}{(J+n)!})\\
    &+  \mathrm{tr}(\sum_{i=0}^{n} (-1)^{i + \sum _{k=0}^{i} |\ka _k|}  
    \sum_{(j_0, ..., j_n)}(-1)^J\frac{\ka_0R^{j_0}\ka'_1\cdots [\Grad,R^{j_{i}}]\ka'_{i+1}R^{j_{i+1}}\cdots \ka'_nR^{j_n}}{(J+n)!}) .
      \end{align*}
    In the last equality, by \eqref{eqn: b2 b1 uB} the first two terms become $\mathrm{tr}_{\nabla}^{II} \circ(b_2+b_1+uB) $
      and by \eqref{eqn: nabla R} the last term becomes  
      \begin{equation}\label{eqn: last1}  -\mathrm{tr}(\sum_{i=0}^{n}j_{i}(d\nu)\sum_{(j_0, ..., j_n)}(-1)^J\frac{\ka_0R^{j_0}\ka'_1\cdots R^{j_{i}-1}\ka'_{i+1}R^{j_{i+1}}\cdots \ka'_nR^{j_n}}{(J+n)!}). \end{equation}
      After letting $K:=J-1$ and  reindexing by $K= k_0 + ... + k_n $, \eqref{eqn: last1} becomes
            \[  d\nu \wedge \mathrm{tr} ( \sum_{(k_0, ..., k_n)}(-1)^K\frac{\ka_0R^{k_0}\ka'_1\cdots R^{k_{i}}\ka'_{i+1}R^{k_{i+1}}\cdots \ka'_nR^{k_n}}{(K+n)!}). \]
            Writing $\nu$ as the sum of $h$ and the curvature element $(\nu - h)$ of $\cP$, we conclude the proof.  \end{proof}

\begin{lemma}
$\tr : \mathrm{End}_A (\cP )    \ot _V \Om _{V/k}[[u]] \to A    \ot _V \Om _{V/k}[[u]] $

\[  (ud  + d_A ) \tr (\gamma ) = \tr [ u\nabla + \delta _{\cP} , \gamma ] \] 
\end{lemma}

\begin{proof}
As a $V$-module, $P$ is a direct summand of a finite sum $N$ of shifted  $V$'s.  
Hence we have $A$-linear maps  $g : P \ot _V  A \to N \ot _V A $ and $f: N \ot _V A  \to P \ot _V A$ such that $f \circ g = 1_{P \ot _V A}$.
Then  we have 
\begin{multline*}   ud   \tr _P (\gamma ) = ud     \tr _N (g \circ  \gamma \circ f  )  =  \tr _N [ u  d   , g\circ  \gamma \circ f  ] \\
 =  \tr _N  (g\circ [ u f \circ d \circ g  ,  \gamma]  \circ f  )   =  \tr _N ( g\circ  [ u\nabla , \gamma ] \circ f )  =\tr _P [ u\nabla  , \gamma ] ,\end{multline*}
 where $\tr _P $, $\tr _N$ indicate the supertrace for $\End _V(P)$, $\End_V (N)$, respectively.
Similarly, $  d_A  \tr _P (\gamma ) = \tr _P [ d_A , \gamma ] = \tr _P [ \delta _{\cP}, \gamma ] $.
\end{proof}

\begin{lemma} With the notation from \eqref{eqn: alpha '}, in the algebra $\End _A (\cP ) \ot _V \Omega ^{\bullet}_V [[u]] $ we have 
 \begin{equation}   \label{eqn: three}
    [u\nabla + \delta _{\cP} , \ka ']   =   [R, \ka ] - ([\delta_{\cP},  \ka ])'  ;                \end{equation}
 \begin{equation} \label{eqn: nabla R}
         [ u \nabla + \delta _{\cP},  R^j ]     =     - j (d\nu) R^{j-1} . \end{equation}
\end{lemma} 
\begin{proof}
The proofs are straightforward. 
\end{proof}

\begin{lemma}  In $A \ot _V \Om _{V/k}[[u]]$ we have
 \begin{multline} \label{eqn: one}   
 \mathrm{tr}\Big(\sum_{i=1}^{n}(-1)^{i-1+ \sum _{k=0}^{i-1} |\ka _k| } \sum_{(j_0, ..., j_n)} (-1)^J \frac{\ka_0R^{j_0}\ka'_1\cdots R^{j_{i-1}}[R,\ka_i]R^{j_i}\cdots \ka'_nR^{j_n}}{(J+n)!}\Big) \\
  =\mathrm{tr}_{\nabla}^{II} \circ b_2(\ka_0[\ka_1 | \dots | \ka_n]);   \end{multline}
  \begin{equation} \label{eqn: two} 
      \mathrm{tr}\big(\sum_{(j_0, ..., j_n)} (-1)^J  
      \frac{u\ka'_0R^{j_0}\ka'_1\cdots \ka '_nR^{j_n}}{(J+n)!}\big)=\mathrm{tr}_{\nabla}^{II} \circ (uB)(\ka_0[ \ka_1 | \dots | \ka_n]); \text{ and }  \end{equation}
      \begin{align}\label{eqn: b2 b1 uB}   &   \mathrm{tr}(\sum_{(j_0, ..., j_n)}(-1)^J  \frac{[\Grad, \ka_0]R^{j_0}\ka'_1\cdots \ka'_nR^{j_n}}{(J+n)!})      \\
&+ \mathrm{tr}(\sum_{i=1}^{n}(-1)^{i-1 + \sum _{k=0}^{i-1} |\ka _k| } 
\sum_{(j_0, ..., j_n)}(-1)^J\frac{\ka_0R^{j_0}\ka'_1\cdots R^{j_{i-1}}[\Grad,\ka'_i]R^{j_i}\cdots \ka'_nR^{j_n}}{(J+n)!}) \nonumber \\
&=  \mathrm{tr}_{\nabla}^{II} \circ(b_2+b_1+uB), \nonumber 
   \end{align} where $\Grad : = u\nabla + \delta _{\cP}$.
  \end{lemma}
\begin{proof}
Equations~\eqref{eqn: one} and~\eqref{eqn: two} are straightforward up to some combinatorics which are checked in the Appendix in \cite{BW}.
Equation~\eqref{eqn: b2 b1 uB} follows from \eqref{eqn: three}, \eqref{eqn: one}, and \eqref{eqn: two}.
\end{proof}

\end{document}